%% file: main.tex
\documentclass[11pt,reqno]{amsart}
\usepackage{amsmath,amssymb,amsthm,tikz-cd,mathrsfs,enumitem,mathtools}
\usepackage[math]{cellspace}
\usepackage{hyperref}
\usepackage[T1]{fontenc}

\numberwithin{equation}{section}

\newtheorem{thm}{Theorem}[section]
\newtheorem{cor}[thm]{Corollary}
\newtheorem{lem}[thm]{Lemma}
\newtheorem{prop}[thm]{Proposition}
\theoremstyle{definition}
\newtheorem{defn}[thm]{Definition}
\theoremstyle{remark}
\newtheorem{rmk}[thm]{Remark}

\newtheorem{conj}[thm]{Conjecture}

\newcommand{\N}{\mathbb{N}}
\newcommand{\Q}{\mathbb{Q}}

\newcommand{\GL}{\operatorname{GL}}

\newcommand{\SL}{\operatorname{SL}}

\newcommand{\SO}{\operatorname{SO}}

\newcommand{\SU}{\operatorname{SU}}

\newcommand{\tr}{\operatorname{tr}}

\renewcommand{\tr}{\operatorname{tr}}

\newcommand{\Trreg}{\operatorname{Tr}_{\operatorname{reg}}}

\newcommand{\vol}{\operatorname{vol}}
\renewcommand{\ker}{\operatorname{ker}}

\newcommand{\Hom}{\operatorname{Hom}}
\newcommand{\End}{\operatorname{End}}
\renewcommand{\dim}{\operatorname{dim}}

\newcommand{\rank}{\operatorname{rank}}

\newcommand{\unip}{\operatorname{unip}}
\newcommand{\geo}{\operatorname{geo}}
\newcommand{\spec}{\operatorname{spec}}
\newcommand{\Ind}{\operatorname{Ind}}

\newcommand{\FP}{\operatorname{FP}}

\newcommand{\mf}[1]{\mathfrak{#1}}

\newcommand{\A}{\mathbb{A}}
\newcommand{\C}{\mathbb{C}}
\newcommand{\Z}{\mathbb{Z}}
\newcommand{\R}{\mathbb{R}}
\newcommand{\F}{\mathbb{F}}
\renewcommand{\L}{\mathcal{L}}

\renewcommand{\O}{\mathcal{O}}

\renewcommand{\phi}{\varphi}
\renewcommand{\epsilon}{\varepsilon}
\renewcommand{\subset}{\subseteq}

\newcommand{\Gal}{\operatorname{Gal}}

\title[Asymptotic behaviour of torsion for $\Q$-rank $1$ arithmetic groups]{Asymptotic behaviour of analytic torsion and cohomological torsion for $\Q$-rank $1$ arithmetic groups}
\author{Tim Berland}
\address{Department of Mathematical Sciences, University of Copenhagen, Universitetsparken 5, 2100 Copenhagen, Denmark}
\email{twb@math.ku.dk}

\begin{document}

\maketitle

\bigskip
\bigskip

\noindent \textbf{Abstract.}
\input{_1.abstract}
\tableofcontents

\pagebreak

\section{Introduction}\label{introduction}
\input{0.introduction}
\bigskip
\bigskip

\clearpage

\thispagestyle{empty}
\vspace*{9pc}
{\huge \part{Asymptotics of analytic torsion}}
\clearpage

\section{Preliminaries}\label{section:preliminaries}
\input{1.preliminaries}

\bigskip

\section{Heat kernel estimates}\label{section:heatkernelestimates}
\input{2-.5.estimates}

\bigskip

\section{Analytic torsion}\label{section:analytictorsion}
\input{2.analytictorsion}
\bigskip

\section{Review of the Arthur-Selberg trace formula}\label{section:traceformula}
\input{3-.5.traceformula}
\bigskip

\section{Restricting to the unipotent contribution}\label{section:compactunipgeneral}
\input{3.restrict}
\bigskip

\section{Explicit local weighted orbital integrals}\label{section:explicitorbitints}
\input{4.localorbital}
\bigskip

\section{Bounds on local orbital integrals}\label{section:boundsorbitints}
\input{5.bounds}
\bigskip

\section{Proof of the approximation theorem}\label{section:proof}
\input{6.proof}
\bigskip

\thispagestyle{empty}
\vspace*{9pc}
{\huge\part{Growth of torsion in the cohomology of arithmetic groups}}
\clearpage

\section{Congruence subgroups of $\SL(2,\mathcal{O}_F)$}\label{section:SL(2,F)}
\input{7.SL2F}

\section{Congruence subgroups of $\SO(n,1)$}\label{section:SO(n,1)}
\input{8.SOn1}

\bigskip

\pagebreak

\bibliographystyle{amsalpha} 
\bibliography{refs} 

\end{document}

%% file: _1.abstract.tex
We extend the refined asymptotics of analytic torsion associated to congruence subgroups of $\SL(n)$ in \cite{berland1}, to congruence subgroups in a large family of reductive groups. This is applied to give new asymptotics and bounds on the growth of torsion in the cohomology of congruence subgroups of $\SL(2,\mathcal{O}_F)$ for $F$ a number field, and of congruence subgroups in $\SO(n,1)$ with $n$ odd.

%% file: 0.introduction.tex
\medskip\noindent For more than a decade, analytic torsion has been a valuable tool in studying growth of torsion in the cohomology of arithmetic groups. In the seminal paper of Bergeron and Venkatesh \cite{BV}, they use analytic torsion to establish exponential growth for an alternating sum of torsion cohomology with certain non-trivial coefficient systems of families of cocompact arithmetic groups. Below we introduce their conjectures and results before moving on to more recent generalizations.

\medskip

\subsection{Conjectures and the cocompact case}

Let $G$ be a connected semisimple algebraic group defined over $\Q$, and let $\Gamma\subset G(\Q)$ be an arithmetic subgroup. A good example to keep in mind is $G=\SL(n)$ and $\Gamma = \SL(n,\Z)$. Let $M$ be a $\Gamma$-module arising from a rational finite-dimensional representation of $G(\R)$. Then the group cohomology $H^*(\Gamma,M)$ has long been understood to have number theoretic significance, such as its link to automorphic forms by the work of Franke \cite{Franke}. If one now considers a lattice $L\subset M$ stabilized by $\Gamma$, one can also consider the cohomology groups $H^*(\Gamma, L)$. These are finitely generated abelian groups, and these, too, carry number theoretic information. In particular, it was conjectured by Ash \cite{Ash}, and later partially proven by Scholze \cite{Scholze}, that there is a correspondence between Hecke eigensystems in $H^*(\Gamma, \F_p)$ and certain Galois representations $\rho:\Gal(\overline{\Q}/\Q)\to \GL(n,\overline{\F}_p)$.

Investigating this cohomology, Bergeron--Venkatesh \cite{BV} gave the following heuristic. Let $K\subset G(\R)$ be a maximal compact subgroup, and define the \textit{deficiency} of $G$ as $\delta(G)\coloneqq \rank_\C G(\R)-\rank_\C K$. Then there is "a lot" of torsion in $H^i(\Gamma, L)$ when $\delta(G)=1$ and $i$ is the middle dimension, and "little" torsion otherwise. To make this precise, they make the following conjecture, generalized in \cite{AGMY}.  

\begin{conj}[Bergeron--Venkatesh]
    Let $\Gamma\eqqcolon \Gamma_0 \supset \Gamma_1\supset\dots\supset \Gamma_N\supset\dots$ be a sequence of congruence subgroups satisfying $\bigcap_N \Gamma_N = \lbrace 0\rbrace$. Then
    \begin{align*}
    \lim_{N\to\infty}\frac{\log|H^i(\Gamma_N,L)_{\text{tors}}|}{[\Gamma:\Gamma_N]}
    \end{align*}
    exists for any $i$, and is equal to a computable constant $c_{\Gamma,L}$. Furthermore, $c_{\Gamma,L}$ is non-zero exactly if $\delta(G)=1$ and $i$ is the top of the cuspidal range.
\end{conj}

\noindent In the paper they prove an Euler-characteristic variant of this (\cite{BV}, $(1.4.2)$). Let $X=G(\R)/K$. Then, under further assumptions, they show that
\begin{align}\label{BVeuler}
    \lim_{N\to\infty} \sum_i (-1)^{i+\frac{\dim X-1}{2}}\frac{\log|H^i(\Gamma_N,L)_{\text{tors}}|}{[\Gamma:\Gamma_N]} = c_{\Gamma,L}.
\end{align}

\noindent Their assumptions include that $\Gamma$ is cocompact in $G(\R)$, equivalently that $G$ has $\Q$-rank $0$, and that $L$ must be a \textit{strongly acyclic} $\Gamma$-module. We will return to these assumptions in a moment.

\pagebreak

It is striking how significant the assumption that $\delta(G)=1$ seems to be, and how precise the conjecture and result are when it is satisfied. On the other hand, they both say very little about what happens when $\delta(G)\neq 1$, except that there is not full exponential growth of torsion. In fact, as phrased in (\cite{AGMY}, §$9$), it is still an open question whether one should expect polynomial or subexponential growth in these cases. Inspired by this question, this paper establishes subexponential upper bounds on the size of torsion in the cohomology of certain families of subgroups. On the way we also establish results on second order terms when $\delta(G) = 1$. The results are obtained for non-cocompact subgroups, and we introduce this generalization now.

\bigskip

\subsection{Generalizations and results on analytic torsion}

Since many important arithmetic subgroups are not cocompact, a lot of consequent work has focused on removing this assumption from (\ref{BVeuler}). This was first successful for arithmetic groups associated to hyperbolic manifolds of finite volume, see \cite{Pfaff2}, \cite{Raimbault2}, \cite{PfR}, \cite{MR2}, and it was recently extended to other $\Q$-rank $1$ arithmetic groups in \cite{MR3}.

All of the results quoted above are consequences of approximation theorems on analytic torsion, combined with \textit{Cheeger-Müller theorems}, identities relating analytic torsion to Reidemeister torsion (see \cite{Müller1},\cite{Cheeger}, also \cite{Müller2},\cite{BZ}). Analytic torsion and Reidemeister torsion are both invariants of Riemannian manifolds, the former built from the spectra of $p$-form Laplacians (see Section \ref{section:analytictorsion}), and the latter from topological data of the manifold (\cite{Reidemeister}, also \cite{Cheeger}). The consequence of Cheeger-Müller theorems is that one may investigate cohomology by analytic methods. 

At the time of the Bergeron--Venkatesh paper, there was no general definition of analytic torsion outside of compact Riemannian manifolds. To remedy this, Matz--Müller gave a definition and the necessary convergence results in \cite{MzM1}, \cite{MzM2}, \cite{MzM3} to show that a non-cocompact generalization of analytic torsion also satisfies the desired approximation theorem, namely (\cite{MzM3}, Theorem $1.5$). 

Our method of getting improved bounds and second order terms on torsion cohomology is sharpening the asymptotics of their approximation theorem. This is the main technical result of this paper, and its preparation and proof will occupy most of the paper. Let us present it here. We adopt the setting from the beginning of the previous section. For simplicity, we keep the assumption that $G$ is semisimple, and present a more general result later (Theorem \ref{mytheoremgeneralized}). We will further assume that $\Gamma_N=\Gamma(N)$, the principal congruence subgroup of level $N$ in $\Gamma$. Assume for convenience that $G$ has strong approximation. Set $\Tilde{X}=G(\R)/K$ and define
\begin{align*}
    X(N) = \Gamma(N)\backslash \Tilde{X}.
\end{align*}
For $N$ large enough, this is then a locally symmetric manifold. Let $d=\dim \Tilde{X}$. For $\tau$ a finite-dimensional complex representation of $G(\R)$ we then set $T_{X(N)}(\tau)$  to be the analytic torsion of $X(N)$ associated to $\tau$, defined in Section \ref{section:analytictorsion}. We also denote by $T_{X(1)}^{(2)}(\tau)$ the $L^2$-torsion of $X(1) = \Gamma\backslash \Tilde{X}$ associated to $\tau$ (\cite{Lott}). 

We will need the notion of a $\lambda$-strongly acyclic representation of $G(\R)$ given in Definition \ref{def:strongly_acyclic}, first introduced in \cite{berland1}. The main theorem is the following.

\pagebreak

\begin{thm}\label{generalizedthmintro}
    Let $\tau$ be an irreducible and $\lambda$-strongly acyclic representation of $G(\R)$, for some $\lambda>0$ only depending on $G$. Assume that $G$ satisfies the properties \emph{(TWN)} and \emph{(BD)} (as defined in \cite{FLM}). Assume either Conjecture \ref{globalcoeffconj} or that $N$ varies over a set with fixed prime divisors (such as $N=p^k$, $k\in\N$). Then there exists some $a>0$ such that
    \begin{align*}
        \frac{\log T_{X(N)}(\tau)}{[\Gamma:\Gamma(N)]} = T^{(2)}_{X(1)}(\tau) + O\left(\frac{(\log N)^a}{N^{k(G)}}\right) \qquad \text{as}\;\;\;N\to\infty. 
    \end{align*}
    Here $k(G)$ is a positive computable integer depending on $G$ defined in (\ref{k(G)}).
\end{thm}

\noindent We refer to Theorem \ref{mytheoremgeneralized} for the more general result.

Conjecture \ref{globalcoeffconj} concerns the size of certain coefficients appearing in the Arthur trace formula. It is believed to hold for all reductive groups, and has been established for $\GL(n)$ over number fields in \cite{Matz}. We note also that the properties (TWN) and (BD) hold quasi-split classical groups (see \cite{FL2} and \cite{FL3}), in particular for $\GL(n)$ and $\SL(n)$, see also the discussion in Section \ref{section:proof}.

We have that $k(\SL(n)) = n-1$ when $\SL(n)$ is defined over $\Q$, as discussed in Section \ref{section:boundsorbitints}, and we also compute $k(G)$ for certain classes of $G$ in Sections \ref{section:SL(2,F)} and \ref{section:SO(n,1)}.

The theorem directly generalizes Theorem $1.1$ of \cite{berland1}. The present proof builds on the work of Matz--Müller \cite{MzM3}, who also establishes the same result with the worse error term $o(1)$ under slightly different assumptions.

This fits into the more general framework of approximation theorems for $L^2$-invariants. We will not pursue this perspective, but refer to \cite{Lück2} for an insightful survey.

\medskip

\subsection{Results on torsion cohomology}

To apply the strengthened approximation theorem above to cohomology, the missing ingredient is a non-compact Cheeger-Müller theorem. The current state of the art in this respect is presented in \cite{MR3} and \cite{MR1}, respectively for arithmetic subgroups in $\SL(2,F)$ for $F$ a number field, and arithmetic subgroups in $\SO(n,1)$ with $n$ odd. These will be the avenues in which we may get results on cohomological torsion.

We adopt first the setting of \cite{MR3}. Let $F$ be a number field of degree $n=r_1+2r_2$, where $r_1$ and $r_2$ are the number of real, respectively pairs of complex embeddings of $F$. For $N\in \N$ set $\Gamma(N)$ the principal congruence subgroup of level $N$ in $\Gamma\coloneqq\SL(2,\mathcal{O}_F)$. For technical reasons, let $G_0=\SL(2)$ defined over $F$, and set $G=\text{Res}_{F/\Q}G_0$ to be its restriction of scalars to $\Q$. We now apply the notation of the previous sections to this setup, in particular the definition of $X(N)$. This is known to satisfy (TWN) and (BD). Let $\rho$ be a $\Q$-rational finite-dimensional representation of $G$.  Let $\rho_\infty$ be the induced representation of $G(\R)$, and let $L_\rho$ be the associated $\Gamma$-module.

\begin{thm}\label{thmcohomsl2intro}
    Assume that $r_2$ is odd, $r_1>0$ and $r_1+r_2>2$. Assume further that $\rho_\infty$ decomposes into a sum of $\lambda$-strongly acyclic representations, with $\lambda>0$ chosen as in Theorem \ref{generalizedthmintro}, and that $L_\rho\otimes \Q$ is acyclic. Then there exists a contant $a>0$ such that
    \begin{align*}
        \frac12\sum_{i=0}^d(-1)^{i+1}\frac{\log|H^i(\Gamma(N),L_\rho)|}{[\Gamma:\Gamma(N)]} = \log T^{(2)}_{X(1)}(\rho_\infty) + O\left(\frac{(\log N)^a}{N^n}\right)
    \end{align*}
    as $N$ tends to infinity.
\end{thm}

\noindent For $r_2>1$ we have that $\delta(G)>1$, and hence we know that $\log T^{(2)}_{X(N)}(\rho_\infty) = 0$. This allows us to give a proper subexponential bound on the growth of torsion.

\begin{cor}\label{mycorintro}
    Assume further that $r_2>1$. Then there exists $a>0$ such that
    \begin{align*}
    \sum_{i=0}^d(-1)^{i+1}\frac{\log|H^i(\Gamma(N),L_\rho)|}{[\Gamma:\Gamma(N)]} = O\left(\frac{(\log N)^a}{N^n}\right)
    \end{align*}
    as $N$ tends to infinity.
\end{cor}

\noindent See Section \ref{section:SL(2,F)} for both of these results.

For our second application, we adopt the setting of \cite{MR1} and \cite{MR2}. Let $n$ be odd, $a_i\in\N$, and let $G$ be the algebraic group over $\Q$ defined by the quadratic form
\begin{align*}
    q(x_1,\dots,x_{n+1}) = a_1x_1^2+\dots+a_nx_n^2-x_{n+1}^2,
\end{align*}
such that $G(\R)\cong \SO(n,1)$. We pick $K$ a maximal compact subgroup, which is then isomorphic to $\SO(n)$. Let $\Gamma = G(\Z)$, and for $N\in\N$ define the principal congruence subgroup
\begin{align*}
    \Gamma(N)\coloneqq \ker(G(\Z)\to G(\Z/N\Z)).
\end{align*}

\noindent We again use the setup above concerning $X(N)$ and $\rho$. 

\begin{thm}\label{thmSO(n,1)intro}
    Let $\rho$ be such that $\rho_\R$ is a $\lambda$-strongly acyclic representation of $G(\R)$ with $\lambda$ as in \ref{generalizedthmintro}, and that $L_\rho\otimes \Q$ is acyclic. Assume (TWN) and (BD) holds for $G$, and assume either Conjecture \ref{globalcoeffconj} or that $N$ varies over a set with fixed prime divisors. Then there exists $a>0$ such that
    \begin{align*}
        \frac12\sum_{i=0}^d(-1)^{q+1}\frac{\log|H^i(\Gamma(N),L_\rho)|}{[\Gamma:\Gamma(N)]} = T^{(2)}_{X(1)}(\rho_\R)+O\left(\frac{(\log N)^a}{N^{k(G)}}\right).
    \end{align*}
    Furthermore, $k(G) = n-2$ when $n\geq 5$, and when $n=3$ we have $k(G) = 2$.
\end{thm}

When $n=3$, we are able to give a more precise result with less assumptions.

\begin{cor}\label{cor:SO(3,1)intro}
    Let $n=3$, such that $G(\R)=\SO(3,1)$. Let $\rho$ be such that $\rho_\R$ is a $\lambda$-strongly acyclic representation of $G(\R)$ with $\lambda$ as in \ref{generalizedthmintro}, and that $L_\rho\otimes \Q$ is acyclic. Then there exists $a>0$ such that
    \begin{align*}
        \frac{\log|H^2(\Gamma(N),L_\rho)|}{[\Gamma:\Gamma(N)]} = -2\cdot T^{(2)}_{X(1)}(\rho_\R)+O\left(\frac{(\log N)^a}{N^2}\right).
    \end{align*}
\end{cor}

\noindent See Section \ref{section:SO(n,1)} for the slightly more precise results and their proofs. 

\medskip

\subsection{Outline of the paper}

Throughout, the paper will differ from the introduction in that it heavily leans on adelic setup and notation, and that the cohomological results are stated for classifying spaces of the congruence subgroups rather than the subgroups themselves.

The paper is split into two parts. Part A is the longer and more technical part, heavily featuring analysis of the Arthur-Selberg trace formula. It gives the setup, preliminary results, and proof of Theorem \ref{generalizedthmintro}. Part B, consisting of two sections, contains the applications to cohomological torsion; Section \ref{section:SL(2,F)} is a proof of Theorem \ref{thmcohomsl2intro} and its corollary, and Section \ref{section:SO(n,1)} is a proof of Theorem \ref{thmSO(n,1)intro}.

We give a more in-depth outline of Part A. In Section \ref{section:preliminaries}, we define the adelic locally symmetric spaces with which we will work, decompose them into locally symmetric manifolds, and define the heat kernel. Section \ref{section:heatkernelestimates} then gives upper bounds on the heat kernel, needed for convergence as well as vanishing of certain error terms. Analytic torsion is defined in Section \ref{section:analytictorsion} using the Arthur-Selberg trace formula, and the trace formula is reviewed in Section \ref{section:traceformula}. In Section \ref{section:compactunipgeneral}, we show that we need only worry about the unipotent contribution to the trace formula. Section \ref{section:explicitorbitints} and \ref{section:boundsorbitints} first makes explicit, and then bounds, the local orbital integrals that constitute the unipotent contribution. Finally, everything is collected and the proof of Theorem \ref{generalizedthmintro} is given in Section \ref{section:proof}.

\medskip

\subsection*{Acknowledgements}

Most of the content of this paper was presented in a preliminary form in the author's PhD thesis. We would like to thank the committee Morten Risager, Werner Müller, and Mehmet Haluk Şengün for their valuable comments and suggestions. We also thank Jasmin Matz for her feedback and helpful discussions.

\medskip

\noindent The work presented here was partially supported by the Carlsberg Foundation, grant CF21-0374.

%% file: 1.preliminaries.tex
\noindent Throughout the paper, we let $G$ be a reductive algebraic group over $\Q$. We write $\A=\A_\Q$ for the ring of adeles of $\Q$, and $\A_f$ for the ring of finite adeles. Let $K_f\subset G(\A_f)$ be any open compact subgroup, and let $K$ be a maximal compact subgroup  in $G(\R)$ with associated Cartan involution $\theta$ on $G(\R)$. We denote by $A_G$ the $\Q$-split component of the center of $G$, giving us a factorization $G(\A)=G(\A)^1\times A_G(\A)^0$. To $K_f$ we associate the arithmetic manifold
\begin{align}\label{adeliclocsym}
    X(K_f) \coloneqq A_G(\R)^0G(\Q)\backslash G(\A)/(K\times K_f),
\end{align}
We give another, more geometric description of $X(K_f)$. Let $x_1,\dots, x_m$ be representatives of the finite double coset space
\begin{align*}
     A_G(\R)^0G(\Q)\backslash G(\A)/G(\R)K_f.
\end{align*}
Then we may define arithmetic groups
\begin{align}\label{arithmsubsfromadelic}
    \Gamma_j \coloneqq (G(\R) \times x_jK_fx_j^{-1})\cap G(\Q), \quad j=1,\dots,m
\end{align}
in $G(\Q)$. Now define the semisimple Lie group $G(\R)^1$ by the factorization $G(\R)=G(\R)^1\times A_G(\R)^0$, and set $\Tilde{X}=G(\R)^1/K$. We have a natural action of $G(\R)$ on the double coset space above, which induces a decomposition of $X(K_f)$ into disjoint spaces
\begin{align*}
    X(K_f) = \bigsqcup_{j=1}^m \Gamma_j\backslash \Tilde{X}.
\end{align*}

\noindent One calls $K_f$ \emph{neat} if the eigenvalues of any element in $\Gamma_j$ for all $j=1,\dots,m$ generate a torsion free subgroup of $\C^\times$. This property guarantees that the spaces $\Gamma_j\backslash \Tilde{X}$ are all locally symmetric manifolds of finite volume. From now on we assume that $K_f$ is neat, so that we may descend to the components $\Gamma_j\backslash \Tilde{X}$, and hence we may make the simplifying assumption, when discussing the differential geometry of $X(K_f)$, that $G$ is connected and semisimple, reducing us to a single component. One then gets the general picture by gluing.

Assume now that $G$ is a connected semisimple algebraic group over $\Q$. We preserve the notation above, and in particular $\Tilde{X}=G(\R)/K$ and $X(K_f) = \Gamma\backslash \Tilde{X}$ for $\Gamma = (G(\R)\times K_f)\cap G(\Q)$. To a finite-dimensional unitary representation of $K$ with inner product $\langle\cdot,\cdot\rangle_\nu$, we associate a homogeneous vector bundle over $\Tilde{X}$ defined as
\begin{align*}
    \Tilde{E_\nu} \coloneqq G(\R)\times_\nu V_\nu.
\end{align*}
The inner product $\langle\cdot,\cdot\rangle_\nu$ induces a $G(\R)$-invariant metric $\Tilde{h}_\nu$ on this vector bundle. Going to the quotient, let $E_\nu\coloneqq \Gamma\backslash\Tilde{E_\nu}$ be the associated locally homogeneous vector bundle over $X(K_f)$, with metric $h_\nu$ induced by $\Tilde{h}_\nu$.

We consider the $K$-equivariant smooth functions
\begin{multline}\label{C-infrepspace}
    \qquad\qquad C^\infty(G(\R),\nu) \coloneqq \lbrace f:G(\R)\to V_\nu \mid f\in C^\infty,\\
    f(gk) = \nu(k)^{-1}f(g) \:\forall k\in K,g\in G(\R)\rbrace.\qquad\qquad\quad\:
\end{multline}
This space is a model of the space of smooth sections of $\Tilde{E}_\nu$, which we denote $C^\infty(\Tilde{X},\Tilde{E}_\nu)$. Indeed, by (\cite{Miatello}, p. 4) there is a $K$-isomorphism
\begin{align}\label{c-inf spaces}
    C^\infty(\Tilde{X},\Tilde{E}_\nu) \xrightarrow{\sim} C^\infty(G(\R),\nu),
\end{align}
\noindent and this extends further to a $K$-isometry of corresponding $L^2$-spaces. 

Let now $(\tau,V_\tau)$ be an irreducible finite-dimensional representation of $G(\R)$, set $E_\tau\coloneqq E_{\tau|_K}$, and let $F_\tau$ be the flat vector bundle over $X(K_f)$ associated to the restriction of $\tau$ to $\Gamma$, viewed as the fundamental group of $X(K_f)$. By (\cite{MM}, Proposition $3.1$), these two vector bundles on $X(K_f)$ are isomorphic. Fix now an inner product on $V_\tau$ with respect to which $\tau|_K$ is unitary. As before, this induces a $G(\R)$-invariant metric on $E_\tau$, and hence on $F_\tau$ as well. 

We use the above setup to understand differential forms. Consider the vector bundle $\Lambda^pT^*(X(K_f))\otimes F_\tau$. Then
\begin{align*}
    \Lambda^p(X(K_f),F_\tau)\coloneqq C^\infty(X(K_f),\Lambda^pT^*(X(K_f))\otimes F_\tau)
\end{align*} 
is the space of $F_\tau$-valued $p$-forms. We write $\Delta_p(\tau)$ for the Laplace operator on $\Lambda^p(X(K_f),F_\tau)$ and $\Tilde{\Delta}_p(\tau)$ for its lift to the universal covering $\Tilde{X}$. Similarly, $\Tilde{F}_\tau$ is the pullback of $F_\tau$ to $\Tilde{X}$. Then $\Tilde{\Delta}_p(\tau)$ is an operator on the space $\Lambda^p(\Tilde{X},\Tilde{F}_\tau)$ defined analogously to the above. Consider the representation $\nu_{\tau,p}\coloneqq\Lambda^p\text{Ad}^*\otimes \tau$ of $K$ on $\Lambda^p\mf{p}^*\otimes V_\tau$. By (\ref{c-inf spaces}), its associated $C^\infty$-space (\ref{C-infrepspace}) is isomorphic to $\Lambda^p(\Tilde{X},\Tilde{F}_\tau)$. This allows us to work with vector-valued $p$-forms and their Laplacian in terms of the representation theory of (\ref{C-infrepspace}).

The operator $\Tilde{\Delta}_p(\tau)$ is formally self-adjoint and non-negative, and thus, regarded as an operator on the space of compactly supported smooth $p$-forms, it has a unique self-adjoint extension to $L^2(\Tilde{X},\Tilde{F}_\tau)$, the $L^2$-sections of $\Tilde{F}_\tau$. This extension inherits non-negativity. By abuse of notation, we also write $\Tilde{\Delta}_p(\tau)$ for the extension. We denote by $e^{-t\Tilde{\Delta}_p(\tau)}$ the heat semigroup associated to $\Tilde{\Delta}_p(\tau)$ for $t>0$. This is a convolution operator when considered as a bounded operator on $L^2(G(\R),\nu_{\tau,p})$, and we denote its kernel $H_t^{\tau,p}$ and view it as a function
\begin{align*}
    H_t^{\tau,p}:G(\R)\to \End(\Lambda^p\mf{p}^*\otimes V_\tau).
\end{align*}
It satisfies the covariance property
\begin{align}\label{covar}
    H_t^{\tau,p}(k^{-1}gk')=\nu_{\tau,p}(k)^{-1}\circ H_t^{\tau,p}(g)\circ \nu_{\tau,p}(k'), \quad \forall k,k'\in K,\: g\in G(\R).
\end{align}
 Let $\mathcal{C}(G(\R))$ denote Harish-Chandra's Schwartz space, and $\mathcal{C}^q(G(\R))$ Harish-Chandra's $L^q$-Schwartz space (see \cite{Wallach}, $7.1.2$). Following the proof of (\cite{BM}, Proposition $2.4$), we have that 
\begin{align}\label{heatkernelspace}
    H_t^{\tau,p}\in \mathcal{C}^q(G(\R))\otimes \End(\Lambda^p\mf{p}^*\otimes V_\tau)
\end{align} 
for any $q>0$. By this description, we see that we may compute its trace over $\Lambda^p\mf{p}^*\otimes V_\tau$ to get a Harish-Chandra Schwarz function. Let 
\begin{align}\label{trheat}
    h_t^{\tau,p}(g)\coloneqq \tr H_t^{\tau,p}(g)
\end{align}
Then $h_t^{\tau,p}\in\mathcal{C}^q(G(\R))$ for any $q>0$, and by (\ref{covar}), we have that it is both left and right $K$-finite. This will be the archimedean part of the test function we use to define analytic torsion in Section \ref{section:analytictorsion}.

%% file: 2-.5.estimates.tex
\noindent In this section, we present the necessary bounds and asymptotics for the trace of the heat kernel. We preserve the assumption that $G$ is semisimple and connected, and continue in the notation of the previous section. In particular, $(\tau,V_\tau)$ is a finite-dimensional irreducible representation of $G(\R)$. Consider $\Tilde{\Delta}_p(\tau)$ as an operator on $C^\infty(G(\R),\nu_{\tau,p})$. It follows from Kuga's lemma that
\begin{align}\label{casimirglobal}
    \Tilde{\Delta}_p(\tau)=\tau(\Omega)-R(\Omega),
\end{align}
where $\Omega$ is the Casimir element of $G(\R)$ and $R$ the action from the right of $G(\R)$ on $C^\infty(G(\R),\nu_{\tau,p})$. Combining this with the Harish-Chandra Plancherel formula, one is able to show exponential decay of $h_t^{\tau,p}$, as done in (\cite{berland1}, Section $4$). The constant appearing in the exponent of the bound is $\tau(\Omega)-\pi(\Omega)$ for $\pi$ an irreducible subrepresentation of $\Lambda^p\mf{p}^*\otimes V_\tau$, and this motivates the following definition, first given in (\cite{berland1}, Definition $3.1$).

\begin{defn}\label{def:strongly_acyclic}
    Let $\lambda>0$. We say that $\tau$ is $\lambda$\textit{-strongly acyclic} if
    \begin{align*}
        \tau(\Omega)-\pi(\Omega) \geq \lambda
    \end{align*}
    for all irreducible unitary representations $\pi$ satisfying $\Hom_K(\Lambda^p\mf{p}\otimes V_\tau^*,\pi)\neq 0$ for some $p$.
\end{defn}

\noindent The \textit{deficiency} of $G$, also sometimes called its fundamental rank, is $\delta(G) = \text{rk}_\C G(\R) - \text{rk}_\C K$. It is shown that there exists infinitely many $\lambda$-strongly acyclic representations for any $\lambda>0$ and for any semisimple $G(\R)$ with $\delta(G)\geq 1$ (see \cite{berland1}, Proposition $3.2$). From now on, we assume that $\delta(G)\geq 1$ and that $\tau$ is $\lambda$-strongly acyclic, with $\lambda>0$ to be determined later. For $g\in G(\R)$, let $r(g)=d(gK,K)$ be the geodesic distance from $K$ to $gK$ on $\Tilde{X}=G(\R)/K$. We present the desired large $t$ asymptotics, summarizing previous work.

\begin{prop}[\cite{berland1}, Proposition $4.2$ and $(4.6)$]\label{heatkerlarget}
Let $\tau$ be $\lambda$-strongly acyclic.
    \begin{itemize}
        \item[(a)] Assume $t\geq 1$. Then there exists a constant $C>0$ such that
        \begin{align*}
            |h_{t}^{\tau,p}(g)|\leq Ce^{-\lambda t}.
        \end{align*}

        \item[(b)] Assume $t \geq 1$. Then there exists constants $A,c,C>0$ such that 
        \begin{align*}
            |h_{t}^{\tau,p}(g)|\leq Ce^{-\lambda t}e^{-c\frac{r(g)^2}{t}}
        \end{align*}
        for all $g\in G(\R)$ with $r(g)>A$.
    \end{itemize}
\end{prop}

\noindent The latter part of the proposition is stronger than the former, and would suffice for our purposes. Part $(a)$ is included merely to simplify and shorten certain proofs where the decay in $r(g)$ is not needed.

%% file: 2.analytictorsion.tex
\noindent We return to the setup of Section \ref{section:preliminaries}, such that in particular $G$ is a reductive group and $K_f$ a neat open compact subgroup of $G(\A_f)$. When needed, we extend the domain of $h_t^{\tau,p}$, which is a priori defined on $G(\R)^1$, to all of $G(\R)$ by setting $h_t^{\tau,p}(ag) = h_t^{\tau,p}(g)$ for all $a\in A_G(\R)$. Furthermore, define
\begin{align}\label{normalizedchar}
    \chi_{K_f} \coloneqq \frac{1_{K_f}}{\vol(K_f)}
\end{align}
for $1_{K_f}$ the characteristic function of $K_f$ on $G(\A_f)$. Let $\mathcal{C}(G(\A)^1,K_f)$ denote the space of smooth right $K_f$-invariant functions on $G(\A)^1$ with all derivatives in $L^1(G(\A)^1)$. Let $h_t^{\tau,p}$ be the trace of the heat kernel as defined above. Then $h_t^{\tau,p}\otimes \chi_{K_f}\in \mathcal{C}(G(\A)^1,K_f)$. Now let $J_{\geo}$ denote the geometric side of the Arthur trace formula. We give a brief review of this in Section \ref{section:traceformula}, see also \cite{Arthur2}, \cite{Arthur0}. The domain of $J_{\geo}$ includes $\mathcal{C}(G(\A)^1,K_f)$ (see \cite{FL}), hence we may define
\begin{align*}
    \Trreg\left(e^{-t\Delta_p(\tau)}\right) \coloneqq J_{\geo}(h_t^{\tau,p}\otimes \chi_{K_f}).
\end{align*}
It was shown by Matz and Müller that this regularized trace has an asymptotic expansion for small $t$ as well as exponential decay for large $t$ (\cite{MzM3}, Theorem $1.1$ and Theorem $1.2$). In particular, by standard theory of Mellin transforms, the zeta function
\begin{align*}
    \zeta_{p,\tau}(s) = \frac{1}{\Gamma(s)}\int_0^\infty \Trreg\left(e^{-t\Delta_p(\tau)}\right)t^{s-1}dt
\end{align*}
is holomorphic in some half plane $\Re(s)\gg 0$ and has a meromorphic continuation to the entire complex plane. We would like to define analytic torsion using its derivative evaluated at $0$, but this may not be well defined if it is has a pole at $0$. Thus, For a meromorphic function $f$ on $\C$ and $s_0\in\C$ we may consider the Laurent expansion at $s=s_0$,
\begin{align*}
    f(s) = \sum_{k\geq k_0}a_k(s-s_0)^k.
\end{align*}
We then define $\text{FP}_{s=s_0}(f) \coloneqq a_0$. Let $d=\dim X(K_f)$. Now we define the analytic torsion $T_{X(K_f)}(\tau)$ of $X(K_f)$ with respect to $\tau$ following \cite{MzM3} as
\begin{align}\label{analytictorsion}
    T_{X(K_f)}(\tau) \coloneqq \frac12\sum_{p=1}^d (-1)^p p \:\text{FP}\left(\frac{\zeta_{p,\tau}(s)}{s}\right).
\end{align}

%% file: 3-.5.traceformula.tex
Fix a minimal parabolic subgroup $P_0$ and a Levi decomposition $P_0=M_0N_0$, and call any subgroup $P\subset G$ containing $P_0$ a \textit{standard} parabolic subgroup of $G$. Any standard parabolic subgroup $P$ has a canonical Levi decomposition, namely $P=MN$ such that $M_0\subset M$. Let $M$ be a Levi subgroup of $G$ over $\Q$, i.e. a Levi component of a parabolic subgroup of $G$. We write $\L(M)$ for the set of Levi subgroups of $G$ containing $M$, and $\mathcal{F}(M)$ for the set of parabolic subgroups containing $M$. Then any $P\in \mathcal{F}(M)$ has a canonical Levi decomposition by picking its Levi component to lie in $\L(M)$. Finally, $\mathcal{P}(M)$ denotes the set of parabolic subgroups for which $M$ is a Levi component. Note that then $\mathcal{F}(M)=\bigsqcup_{L\in\mathcal{L}(M)}\mathcal{P}(L)$. These are all finite sets. Throughout, we will pick the Haar measures on our groups such that $K$, $G(\Z_p)$ for each prime $p$, and $N_P(\Q)\backslash N_P(\A)$ all have measure $1$, and they respect the decompositions $G(\A) = P(\A)K = M(\A)N(\A)K = M(\A)^1N(\A)A(\R)^0K$.

Denote by $A_M$ the $\Q$-split component of the center of $M$. We will sometimes write $A_P$ for $A_{M_P}$ when $M_P$ is the canonical Levi component of $P$. We further denote by $X(M)_\Q$ the group of characters of $M$ defined over $\Q$, and set $\mf{a}_M= \Hom(X(M)_\Q, \R).$ Again, we may write $\mf{a}_P \coloneqq \mf{a}_{M_P}$. This is then a real vector space.

The Arthur trace formula is the equality
\begin{align*}
    J_{\geo}(f) = J_{\spec}(f)
\end{align*}
of the two distributions $J_{\geo}$ and $J_{\spec}$ on $C_c^\infty(G(\A)^1)$, known as the geometric side, respectively the spectral side, of the trace formula. This was introduced in \cite{Arthur2}, and has since been expanded on immensely, see \cite{Arthur0} for a great overview and introduction. Here, we will only briefly and without proof introduce the coarse and fine expansions of the geometric side.

\subsection{The coarse geometric expansion}\label{coarsegeomintro}

Define an equivalence relation on $G(\Q)$ by setting two elements equivalent if their semisimple parts are $G(\Q)$-conjugate to each other. We denote by $\mathcal{O}$ the set of such equivalence classes in $G(\Q)$. Note that these are in bijection with conjugacy classes of semisimple elements. It is an elementary result that $\mathcal{O}$ is finite. For any such equivalence class $\mf{o}\in\O$, we have an associated distribution $J_{\mf{o}}$ giving a decomposition of the geometric side called the \textit{coarse geometric expansion}
\begin{align}\label{eq:coarsegeom}
    J_{\geo}(f) = \sum_{\mf{o}\in\mathcal{O}}J_{\mf{o}}(f),
\end{align}
\noindent See \cite{Arthur2} for the proper definition and the sum formula above. We give a working description of these distributions $J_{\mf{o}}$ for a class of equivalence classes below. For our purposes, especially one equivalence class will be important, namely the one with semisimple part being the identity, which is of course the unipotent elements $\mf{o}_{\unip}$. For brevity we will write $J_{\unip}=J_{\mf{o}_{\unip}}$ going forward.

First, let $\mf{o}\in \O$ be anisotropic, and let $P$ be a certain associated standard parabolic subgroup with $\gamma\in P\cap \mf{o}$ semisimple (anisotropic root datum, see \cite{Arthur0}, §$10$). Write $P=MN$ the canonical Levi decomposition of $P$ . Then
\begin{align}\label{eqclassdist}
    J_{\mf{o}}(f)=\vol(M(\Q)_\gamma\backslash M(\A)_\gamma^1) \int_{G(\A)_\gamma\backslash G(\A)}f(x^{-1}\gamma x)v_P(x) dx.
\end{align}

\noindent The function $v_P(x)$ here is a certain weight function, see (\cite{Arthur0}, Theorem $11.2$) for a definition, and the result above. Note that not all classes $\mf{o}\in\O$ are anisotropic, in particular the formula does not hold for $J_{\unip}$. Rather, we provide a further decomposition of this distribution which will allow us to express it in terms of local orbital integrals.

\subsection{The fine geometric expansion}\label{finegeomintro}

Let $S$ be a finite set of places including $\infty$, and let $M$ be a Levi subgroup. Choose a $\gamma=\prod_{\nu\in S}\gamma_\nu\in M(\Q_S)$. The centralizer $G_{\gamma_\nu}$ is an algebraic group over $\Q_\nu$, and we define $G_\gamma = \prod_{\nu\in S}G_{\gamma_\nu}$. It follows from the short paper of Rao (\cite{Rao}) that there exists a right-$G(\Q_S)$-invariant Radon measure on the quotient $G_\gamma(\Q_S)\backslash G(\Q_S)$, meaning we have a $G(\Q_S)$-invariant linear form taking $f\in C_c^\infty(G(G_S))$ to
\begin{align}\label{localorbitalint}
    \int_{G_\gamma(\Q_S)\backslash G(\Q_S)} f(x^{-1}\gamma x) dx.
\end{align}

\noindent We will now assume $G_\gamma=M_\gamma$ to simplify the description. At the end, we explain how to go to the general case. Recall the definition of the generalized Weyl discriminant,
\begin{align*}
    D(\gamma) &= \prod_{\nu\in S}D(\gamma_\nu), \\
    D(\gamma_\nu) &=\det(1-\text{Ad}((\gamma_\nu)_s))_{\mf{g}/\mf{g}_{(\gamma_\nu)_s}},
\end{align*}

\noindent where $(\gamma_\nu)_s$ is the semisimple part of $\gamma_\nu$, and $\mf{g}_{\sigma_\nu}$ the Lie algebra of $G_{\sigma_\nu}$. This discriminant is used for normalization purposes. Then we define the weighted orbital integral
\begin{align}\label{orbintintro}
    J_M(f,\gamma)=|D(\gamma)|^{\frac12}\int_{G_\gamma(\Q_S)\backslash G(\Q_S)}f(x^{-1}\gamma x)v_M(x) dx.
\end{align}

\noindent Here $v_M(x)$ is another weight function related to $v_P$, see (\cite{Arthur0}, $(18.2)$).

If $G_\gamma\neq M_\gamma$, we have a more complicated formula. For $a\in A_M(\Q_S)$ small in general position, there exists canonical functions $r^L_M(\gamma,a)$ such that the following exists and is well defined:
\begin{align*}
    J_M(f,\gamma) = \lim_{a\to 1}\sum_{L\in\mathcal{L}(M)}r^L_M(\gamma,a)J_L(f,a\gamma).
\end{align*}

\noindent This is one of the central theorems of (\cite{Arthur1}), see also there for the relevant existence and convergence results. Considering the obvious injection of $C_c^\infty(G(\Q_S)^1)$ into $C_c^\infty(G(\A)^1)$, we may form the distribution $J_\mf{o}(f)$ for $f\in C_c^\infty(G(\Q_S)^1)$. We need a finer equivalence relation to index over.

\begin{defn}
    We say two elements $\gamma_1,\gamma_2\in M(\Q)\cap \mf{o}$ are $(G,S)$-\textit{equivalent} if they are $G(\Q_S)$-conjugate. The set of such equivalence classes $(M(\Q)\cap \mf{o})_{G,S}$ is then finite.
\end{defn}

\noindent Finally, we have the characterization of $J_{\mf{o}}(f)$ in terms of weighted orbital integrals.

\begin{thm}[\cite{Arthur7}]\label{generalorbdistributionintro}
    Let $f\in C_c^\infty(G(\Q_S)^1)$. There exists a finite set of places $S_\mf{o}$, only depending on $\mf{o}$, containing $\infty$ and such that for any finite $S$ with $S_\mf{o}\subset S$, we have
    \begin{align*}
        J_{\mf{o}}(f)=\sum_{\substack{M\in\mathcal{L}, \\\gamma\in (M(\Q)\cap \mf{o})_{M,S}}}a^M(S,\gamma)J_M(f,\gamma),
    \end{align*}
    where $a^M(S,u)$ are some uniquely determined coefficients.
\end{thm}

\noindent The coefficients $a^M(S,\gamma)$, sometimes called the \emph{global coefficients}, are difficult to compute in general. In the special case $\gamma=1$, we have
\begin{align*}
    a^M(S,1) = \vol(M(\Q)\backslash M(\A)^1).
\end{align*}

\noindent Furthermore, for $\GL(n)$ defined over a number field $F$, a bound has been established in \cite{Matz} for the global coefficients associated to unipotent $\gamma$ which we will discuss later.

Assume now that $f\in C_c^\infty(G(\A))$ can be expressed as a product $f=\prod_\nu f_\nu$ with $f_\nu\in C_c^\infty(G(\Q_\nu))$ and all but finitely many $f_p = 1_{K_p}$. We would like to split up the weighted orbital integrals into local parts, using that the weight functions also satisfy only a decomposition into local functions, however an additive one. This gives a formula involving sums of products of local parts. Handling this in a smart way yields the following result (see \cite{Arthur3}).

Assume $S=S_1\sqcup S_2$, and $f=f_1f_2$ with $f_i\in C_c^\infty(M(\Q_{S_i}))$ along with $\gamma=\gamma_1\gamma_2$ such that $\gamma_i\in M(\Q_{S_i})$. We assume $G_\gamma=M_\gamma$, implying that $G_{\gamma_i}=M_{\gamma_i}$. For $Q_i=MN_i\in \mathcal{P}(M)$ define
\begin{align}\label{twisttestfunction}
    f_{i,Q_i}(m) = \delta_{Q_i}(m)^{\frac12}\int_{K_{S_i}}\int_{N_i(\Q_{S_i})}f_i(k^{-1}mnk)dndk.
\end{align}
Here $\delta_Q(m)$ is a certain normalization independent of $f$ (see e.g. \cite{Arthur0}, p. $50$). Then we have a general decomposition formula given by

\begin{align}\label{eq:finegeom}
    J_M(f,\gamma)=\sum_{L_1,L_2\in\mathcal{L}(M)}d^G_M(L_1,L_2)J_M^{L_1}(f_{1,Q_1},\gamma_1)J_M^{L_2}(f_{2,Q_2},\gamma_2).
\end{align}

\noindent The orbital integral $J_M^L$ is the analogous expression where one instead considers the underlying reductive group to be $L$, instead of $G$. The constant $d^G_M(L_1,L_2)$ arises from the decomposition formula of the weight functions. Applying the formula inductively let's us reduce to sets $S_i$ of one element. It is the combination of Theorem \ref{generalorbdistributionintro} along with this decomposition, used on the terms $J_{\mf{o}}(f)$ on the geometric side of the Arthur-Selberg trace formula, that we will call the \textit{fine geometric expansion}. When $G_\gamma\neq M_\gamma$, the same procedure yields an expansion in limits of sums of the terms above.

%% file: 3.restrict.tex
\noindent We continue in the setup of Section \ref{section:analytictorsion}. In this section we show that, when compactifying your test function properly, one needs only concern themselves with the unipotent contribution to the trace formula. The result below is a generalization of (\cite{berland1}, Proposition $5.3$), which is the equivalent statement for $G=\GL(n)$, and the proof in fact reduces to this result after choosing a suitable embedding of $G$ into $\GL(n)$.

For now, we provide the setup. Let $G$ be a connected reductive group defined over $\Q$ with neat open compact subgroup $K_f = \prod_p K_p\subset G(\A_f)$. Take a faithful $\Q$-rational representation $\rho:G\to\GL(V)$ with a lattice $\Lambda\subset V$ such that $K_f$ is the stabilizer in $G(\A_f)$ of $\hat{\Lambda}\coloneqq \hat{\Z}\otimes \Lambda\subset \A_f\otimes V$. For $N\in\N$ we define
\begin{align}\label{adelicprinccongsub}
    K(N)=\lbrace g\in G(\A_f)\mid \forall v\in \A_f\otimes V:\:\rho(g)v\equiv v \mod N\hat{\Lambda}\rbrace.
\end{align}
Then $K(N)\subset K_f$ is a factorizable open compact subgroup in $G(\A_f)$. We see that these generalize the open compact subgroups of $\GL(n)$ given as 
\begin{align*}
    K_{\GL(n)}(N) \coloneqq \prod_p \ker(\GL(n,\Z_p)\to \GL(n,\Z_p/p^e\Z_p)),
\end{align*}
by picking $\rho$ the standard representation and $\Lambda$ the $\Z$-span of the standard basis. Furthermore, given $\Lambda\subset V$ we may identify $\GL(V)$ with $\GL(n)$ by choosing a $\Z$-basis of $\Lambda$ as a basis of $V$, and under this identification, $K(N)$ maps into $K_{\GL(n)}(N)$ under the embedding $\rho$. 

We let $h_t^{\tau,p}:G(\R)^1\to \R$ be the trace of the heat kernel as defined in (\ref{trheat}) and $\chi(N)$ the normalized characteristic function on $G(\A_f)$ of $K(N)$ as in (\ref{normalizedchar}). We then take $f=h_t^{\tau,p}\otimes \chi(N)\in\mathcal{C}(G(\A)^1,K(N))$ as our test function, in the same way we did when we defined analytic torsion. Although this function is not compactly supported, the geometric side of the Arthur-Selberg trace formula as well as the coarse geometric expansion are well defined by (\cite{FL}). 

Set $r(g)=d(K,gK)$ to be the geodesic distance on $G(\R)^1/K$ of $g\in G(\R)^1$ from the identity, and let $\phi_R:G(\R)^1\to [0,1]$ be any smooth function identically $1$ on $B(R)=\lbrace g\in G(\R)^1\mid r(g)<R\rbrace$ and identically $0$ outside $B(R+\epsilon)$ for some $\epsilon>0$ small. We then define the compactification of $h_t^{\tau,p}$ with parameter $R$ as
\begin{align}\label{compactificationgeneral}
    h_{t,R}^{\tau,p}(g) = \phi_R(g)h_t^{\tau,p}(g), \quad g\in G(\R)^1.
\end{align}

\noindent Consider the embedding $G(\R)^1\to \GL(n,\R)^1$ induced by $\rho$ as above. By possibly composing with an inner automorphism of $\GL(n,\R)^1$, we can ensure the following properties hold (see also \cite{MzM3}, p. $24$):
\begin{enumerate}
    \item $\rho(K)\subset \text{O}(n)$,
    \item If $\mf{g}=\mf{k}\oplus\mf{p}$ is the Cartan decomposition of the Lie algebra $\mf{g}$ of $G(\R)^1$ associated to $\theta$, then $\mf{k}$ (resp. $\mf{p}$) is embedded into the skew-symmetric (resp. symmetric) matrices of $\mf{gl}(n,\R)$ under the induced map,
    \item the Frobenius norm on the embedding of $\mf{g}$ in $\mf{gl(n,\R)}$ coincides with the norm induced by the Killing form.
\end{enumerate}
It follows from such a choice of embedding that
\begin{align}\label{distancesagree}
    d(gK,K) = d_{\GL(n)}(\rho(g)\text{O}(n),\text{O}(n))
\end{align}
where $d_{\GL(n)}$ is the geodesic distance on $\GL(n,\R)^1/\text{O}(n)$. We now state and prove our reduction of the geometric side of the trace formula to the unipotent contribution.

\begin{prop}\label{compactunipidentitygeneral}
    There exists a constant $C_n>0$ only depending on $n$ such that for $N$ large enough and $R\leq C_n \log N$, we have that
    \begin{align*}
        J_{\geo}(h_{t,R}^{\tau,p}\otimes \chi(N)) = J_{\unip}(h_{t,R}^{\tau,p}\otimes \chi(N)).
    \end{align*}
\end{prop}

\begin{proof}
    Let $\mathcal{O}$ be defined as in Section \ref{coarsegeomintro}, and let $\mf{o}\in \mathcal{O}$ with representative $\gamma$. Consider any $f\in \mathcal{C}(G(\A)^1,K(N))$. It follows directly from the formula (\ref{eqclassdist}) that $J_{\mf{o}}(f)=0$ if no conjugate of $\gamma$ in $G(\A)$ intersects the support of $f$.

    Let now $\mf{o}\neq \mf{o}_{\text{unip}}$. The orbit of $\gamma$ under $G(\A)$ maps into the orbit of $\rho(\gamma)$ in $\GL(n,\A)$. Furthermore, under $\rho$, the ball $B(R)\subset G(\R)^1$ of radius $R$ maps into the ball $B_{\GL(n)}(R)$ of radius $R$ under geodesic distance in $\GL(n,\R)^1$ by (\ref{distancesagree}), and $K(N)$ maps into $K_{\GL(n)}(N)$ by construction as detailed above.

    It was shown in (\cite{berland1}, Proposition $5.3$) that the $\GL(n,\A)$-orbit of any such $\rho(\gamma)$ does not intersect $B_{\GL(n)}(R)K_{\GL(n)}(N)$ for $N$ large enough and $R\leq C_n\log N$, with $C_n>0$ a constant depending only on $n$. By the above, this implies that the $G(\A)$-orbit of $\gamma$ does not intersect $B(R)K(N)$, which is the support of $h_{t,R}^{\tau,p}\otimes \chi(N)$. In particular, $J_{\mf{o}}(h_{t,R}^{\tau,p}\otimes \chi(N)) = 0$. By the coarse geometric expansion (\ref{eq:coarsegeom}), this completes the proof.
\end{proof}

%% file: 4.localorbital.tex
\noindent In order make full use of the fine geometric expansion of $J_{\unip}$, we need to have an explicit formula for the local orbital integrals appearing in (\ref{eq:finegeom}). This is more complicated than in the case of $\GL(n)$ as we do not in general have access to Richardson parabolic subgroups, and instead must use Jacobson-Morozov parabolics. The explicit formulas will be used both in the archimedean and non-archimedean setting in the following sections.

We continue in the same notation. In particular, $G$ will be a connected reductive group over $\Q$ and $\nu\in\lbrace p,\infty\rbrace$ a place of $\Q$. We let $K_\nu$ be an admissible maximal compact subgroup of $G(\Q_\nu)$ as defined in (\cite{Arthur3}, §$1$). Let $M$ be a Levi subgroup of $G$ with Lie algebra $\mf{m}$, and take $\gamma\in M(\Q_\nu)$ a representative of a unipotent conjugacy class, i.e. $\gamma_{ss} = 1$. For $P\in \mathcal{F}(M)$ we write $P=M_PN_P$ for its Levi decomposition with $M\subset M_P$. 

By the Jacobson-Morozov theorem there exists an $\mf{sl}_2$-triple $(X,Y,H)$ in $M(\Q_\nu)$ such that $\gamma = \exp X$. Let $\mf{m}(\Q_\nu)=\bigoplus_{i\in\Z}\mf{m}_i$ be the associated eigenspace decomposition, and define
\begin{align*}
    Z_\nu \coloneqq \bigg\lbrace p^{-1}\gamma p \mid p \text{ in the normalizer in }M(\Q_\nu)\text{ of }\bigoplus_{i\geq 0}\mf{m}_i\bigg\rbrace.
\end{align*}

\noindent It is known that $Z_\nu$ is an open subset of $\exp \mf{u}_2$, for $\mf{u}_k\coloneqq \bigoplus_{i\geq k}\mf{m}_i$ (see \cite{Arthur1}, p. $246$). Fix $Q\in\mathcal{F}(M)$ with Levi decomposition $Q=LN_Q$, and take some $R\in \mathcal{P}(M)$ contained in $L$. Define $\Pi_\nu^Q = Z_\nu N_R(\Q_\nu)$. With Haar measures on $K_\nu$, $N_P(\Q_\nu)$, and the product of Haar measures on $\Pi_\nu^Q$, the proof of (\cite{Arthur1}, Corollary $6.2$) gives the local weighted orbital integrals as

\begin{align*}
    J_M(f,\gamma)=\sum_{Q\in\mathcal{F}(M)}c(Q,\gamma)\int_{K_\nu}\int_{N_Q(\Q_\nu)}\int_{\Pi^Q_\nu} f(k^{-1}\pi nk)v^Q_M(1,\pi)|I^Q_\nu(\pi)|_\nu^{\frac12}d\pi dn dk.
\end{align*}

\noindent Our integral is significantly simpler than that of \cite{Arthur1} due to the fact that we assume $\gamma$ is unipotent. In the above, $I^Q_\nu$ is a polynomial on $\mf{u}_2(\Q_\nu)$ of $\deg I^Q_\nu=\dim \mf{m}_1$ defined over $\Q_\nu$ (see \cite{Rao}), and $v^Q_M(1,\pi)$ is a weight function defined in (\cite{Arthur1}, §$5$), similar to that of Section \ref{section:traceformula}. In particular, the weight function satisfies log-homogeneity, i.e. it is bounded by some polynomial in the norm of their entries.  

To simplify the integral, we will instead integrate over the associated Lie algebras. Let $\mf{n}_Q$ be the Lie algebra of $N_Q(\Q_\nu)$, and similarly for $\mf{n}_R$. Let $\mf{v}_Q = \mf{n}_Q\oplus \mf{u}_2\oplus\mf{n}_R$, which is a nilpotent Lie algebra. Using that $Z_\nu$ is dense in $\exp\mf{u}_2$, we can extend both the weight function and the polynomial $I^Q_\nu$ to $\mf{n}_Q\oplus\mf{u}_2\oplus\mf{n}_R$ by projecting to the relevant subspace. 

\begin{align}\label{explicitlocalorbint}
    J_M(f,\gamma)=\sum_{Q\in\mathcal{F}(M)}c(Q,\gamma)\int_{K_\nu}\int_{\mf{v}_Q} f(k^{-1}e^{X} k)v^Q_M(1,e^{X})|I^Q_\nu(X)|_\nu^{\frac12}dX dk.
\end{align}

\noindent We can relate the dimensions of these Lie algebras with the dimensions of orbits of $X$. Let $\mathcal{O}_X$ be the orbit of $X$ in $M$. We then have the following identity of dimensions (see \cite{CM}, Lemma $4.1.3$ for the complex version, the general proof is analogous),
\begin{align*}
    \dim \mathcal{O}_X = \dim \mf{m}-\dim \mf{m}_0-\dim \mf{m}_1 = 2\dim \mf{u}_2+\dim \mf{m}_1,
\end{align*}
where we have used that $\dim \mf{m}_k = \dim \mf{m}_{-k}$ for every $k\in\Z$. 

Furthermore, by a standard formula for dimensions of induced orbits (\cite{CM}, Theorem $7.1.1$ for the complex version), applied twice using that $\Ind_M^G \mathcal{O}_X = \Ind^G_L \Ind^L_M\mathcal{O}_X$, we have that
\begin{align*}
    \dim \Ind^G_M \mathcal{O}_X &= \dim \Ind^L_M \mathcal{O}_X +2\dim \mf{n}_Q \\
    &= \dim \mathcal{O}_X +2\dim \mf{n}_R +2\dim \mf{n}_Q \\
    &= 2(\dim \mf{u}_2+\dim\mf{n}_R+\dim\mf{n}_Q)+\dim\mf{m}_1 \\
    &= 2\dim \mf{v}_Q+\dim\mf{m}_1.
\end{align*}

\noindent Let us end this section by mentioning an explicit description of the weight functions. 

\begin{lem}[\cite{MzM3}, Lemma $5.4$]\label{weightfunctionsinMzM3}
    There exists constants $r,t\geq 0$ and polynomials $q_1,\dots,q_r: \mf{v}_Q\to \Q_\nu$ and complex polynomials $R_1,\dots,R_t$ in $r$ many variables such that for any $X\in\mf{v}_Q$ and $s>0$, there exists some $C>0$ such that
    \begin{align*}
        v_M^Q(1,I+sX) = \sum^t_{i=0} \log(s)^iR_i(\log|q_1(X)|_\nu, \dots, \log|q_r(X)|_\nu).
    \end{align*}
\end{lem}

\noindent Note that the proof of (\cite{MzM3}, Lemma $5.4$) was shown in the case of $\nu=\infty$, but it holds in general with only minor modifications. It follows immediately that we have an upper bound on the weight functions. There exists some $b>0$ such that
\begin{align}\label{localweightbound}
    |v_M^Q(1,e^X)|\leq C (1+\log \lVert e^X\rVert)^b.
\end{align}
The constant $C>0$ can be chosen to be independent of $M$ and $Q$.

%% file: 5.bounds.tex
\noindent This section utilizes the general explicit description of local orbital integrals given in the previous section to give the necessary upper bounds for our test function. We treat the archimedean and non-archimedean cases separately.

\subsection{Large $t$ asymptotics of archimedean orbital integrals}

We focus here on the infinite place $\nu=\infty$, and will critically use our bounds on $h_t^{\tau,p}$ from Section \ref{section:heatkernelestimates}.

we choose an embedding of $G(\R)^1$ in $\GL(n,\R)^1$ as in the paragraph before (\ref{distancesagree}). Recall that the geodesic distance function on $\GL(n,\R)^1/\text{O}(n)$ then agrees with the geodesic distance function on $G(\R)^1/K$. 

\begin{prop}
Let $\gamma\in M(\R)$ be unipotent, and assume $t> 1$ and that $\tau$ is $\lambda$-strongly acyclic. Then there exists some $C,c'>0$, both independent of $R$ and $c'$ independent of $\lambda$, such that
    \begin{align}\label{archbound}
        |J_M^G(h_{t,R}^{\tau,p},\gamma)|
        &\leq C e^{-(\lambda-c') t}.
    \end{align}
\end{prop}

\begin{proof}
    By the formula in (\ref{explicitlocalorbint}) and by using the bi-$K$-invariance of $h_{t,R}^{\tau,p}$, we are reduced to considering
    \begin{align}\label{archintermed}
        J_M^G(h_{t,R}^{\tau,p},\gamma)=\sum_{Q\in\mathcal{F}(M)}c(Q,\gamma)\int_{\mf{v}_Q} h^{\tau,p}_{t,R}(e^X)v^Q_M(1,e^{X})|I^Q_\infty(X)|^{\frac12}dX .
    \end{align}
    Let $A,c,C>0$ be as in Proposition \ref{heatkerlarget}(b), and write
    \begin{align*}
        \mf{v}_Q = \bigcup_{k=0}^\infty D(k),
    \end{align*}
    where
    \begin{align*}
        D(0)&\coloneqq \lbrace X\in\mf{v}_Q\mid r(e^X)<A\rbrace, \\
        D(k) &\coloneqq \lbrace X\in\mf{v}_Q\mid A+k\leq r(e^X)<A+k+1\rbrace, \quad k\geq 1.
    \end{align*}
    We will decompose the orbital integral (\ref{archintermed}) into a sum of integrals using this disjoint union. Throughout the following, we use the bound (\ref{localweightbound}), which implies the bound
    \begin{align*}
        |v_M^Q(1,e^X)|\leq C \left(1+r(e^X)\right)^b.
    \end{align*}
    We will also use that the polynomial $I^Q_\infty(X)$ is bounded by some polynomial in $\lVert X\rVert$, and as $\lVert X\rVert$ is bounded by $C(1+\lVert e^X\rVert)$ for some constant $C>0$ for all unipotent matrices (see e.g. \cite{MzM3}, Lemma $6.2$), it is furthermore bounded by a polynomial in $e^{r(e^X)}$. 
    The integral over $D(0)$ is then handled using Proposition \ref{heatkerlarget}(a):
    \begin{align*}
        \left\vert\int_{D(0)} h^{\tau,p}_{t,R}(e^X)v^Q_M(1,e^{X})|I^Q_\infty(X)|^{\frac12}dX\right\vert &\leq C'e^{-\lambda t}\int_{D(0)} \left(1+r(e^X)\right)^b\left(1+e^{r(e^X)}\right)^{\frac12}dX \\
        &\leq C' e^{-\lambda t}\left(1+A\right)^b \left(1+e^A\right)^{\frac12} \vol(D(0)).
    \end{align*}

    \noindent As $D(0)$ is compact, its volume only depending on $A$, we see that the above is $C_A e^{-\lambda t}$, for some constant $C_A>0$ only depending on $A$. 

    By the bound on $\lVert X\rVert$ in terms of $\lVert e^X\rVert$ above, we may bound the volume of $D(k)$ by the volume of
    \begin{align*}
        \lbrace X\in \mf{v}_Q \mid \lVert X\rVert \leq C(1+e^k)\rbrace
    \end{align*}
    As this is a ball in Euclidean space, we get that
    \begin{align*}
        \vol(D(k)) \leq C(1+e^{k})^{\dim \mf{v}_Q}.
    \end{align*}
    We may now estimate the integral over $D(k)$. First, use Proposition \ref{heatkerlarget}(b) to write
    \begin{align*}
        \left\vert\int_{D(k)} h^{\tau,p}_{t,R}(e^X)v^Q_M(1,e^{X})|I^Q_\infty(X)|^{\frac12}dX\right\vert \leq C'e^{-\lambda t}\left\vert\int_{D(k)} e^{-c\frac{r\left(e^X\right)^2}{t}}v^Q_M(1,e^{X})|I^Q_\infty(X)|^{\frac12}dX\right\vert.
    \end{align*}
    By the above, the latter integral can then be bounded by
    \begin{align*}
        C' e^{-c\frac{(A+k-1)^2}{t}}\left(1+(A+k)\right)^b \left(1+e^{A+k}\right)^d \vol(D(k)) \\
        \leq  C_A'e^{-c\frac{k^2}{t}}e^{c_1k},
    \end{align*}
    where $C_A',d>0$ are some constants and $c_1 = \dim\mf{v}_Q+d+\epsilon$ for some (any) small $\epsilon>0$. It now follows as in the end of the proof of (\cite{berland1}, Proposition $6.4$) that the sum of the integrals over $D(k)$, $k\geq 1$, converges and is bounded by 
    \begin{align*}
        C_2 e^{-(\lambda -c')t}
    \end{align*}
    for some constants $C_2,c'>0$, with $c'$ only depending on the group $G$. This finishes the proof.
\end{proof}

\bigskip

\subsection{Explicit bounds on non-archimedean orbital integrals}

Our setup is the following. We will stick to the notation of $K(N)$ and $\chi(N)$ as defined in (\ref{adelicprinccongsub}). Write $K(N) = \prod_{p}K_p(p^{\nu_p(N)})$ for the factorization of $K(N)$ such that $K_p(p^{\nu_p(N)})$ is mapped into $\ker(\GL(n,\Z_p)\to \GL(n,\Z_p/p^{\nu_p(N)}\Z_p))$ under the embedding $G\to\GL(n)$. We let $M$ be a Levi subgroup of $G$, and $P\in\mathcal{F}(M)$ with Levi decomposition $P=LN_P$. Let $\mathcal{U}_M(\Q)$ denote the set of unipotent conjugacy classes in $M(\Q)$, and write $\gamma=\prod_\nu \gamma_\nu$ for $\gamma$ a representative of $\mathcal{O}_\gamma\in \mathcal{U}_M(\Q)$, considered as an element of $M(\A)$. 

We now prove a bound on the non-archimedean part of $J_{\unip}$. We begin by working locally. Fix $p$ and write $e=\nu_p(N)$. Using the explicit formula for local orbital integrals (\ref{explicitlocalorbint}) on the reductive group $L$, unfolding the definition of $f_P$ as given in (\ref{twisttestfunction}), and using the fact that $K_p(p^e)$ is a normal subgroup, we may write
\begin{align}\label{explicitlocalnonarch}
    J_M^L((1_{K_p(p^e)})_P,\gamma)=\sum_{Q\in\mathcal{F}^L(M)}c(Q,\gamma)\int_{\mf{v}_Q} 1_{K_p(p^e)}(e^{X})v^Q_M(1,e^{X})|I^Q_p(X)|_p^{\frac12}dX.
\end{align}

\noindent Considering the characteristic function, we may restrict to computing this integral over $\lbrace X\in \mf{v}_Q\mid e^X\in K_p(p^e)\rbrace$, which has a norm- and measure-preserving isomorphism to $p^e\Z_p^{\dim \mf{v}_Q}$ equipped with the Haar measure from $\Q_p$. Using the bound on weight functions (\ref{localweightbound}) and that $I^Q_p$ is homogeneous of degree $\dim \mf{m}_1$, we may estimate the integral on the right hand side above as

\begin{align*}
    \left\vert\int_{\mf{v}_Q} 1_{K_p(p^e)}(e^{X})v^Q_M(1,e^{X})|I^Q_p(X)|_p^{\frac12}dX\right\vert &\leq C\int_{p^e\Z_p^{\dim\mf{v}_Q}} (1+\log\lVert e^X\rVert)^b \lVert X\rVert^{\frac{\dim \mf{m}_1}{2}} dX \\
    &\leq C' p^{-e\dim\mf{v}_Q}(1+\log(p^{-e}))^b p^{-\frac{e\dim\mf{m}_1}{2}} \\
    &\leq C' p^{-\frac{e\dim\Ind_M^G\mathcal{O}_\gamma}{2}}(1+\log(p^{-e}))^b .
\end{align*}

\noindent Here we also used the description of the dimension of orbits provided above Lemma \ref{weightfunctionsinMzM3}. In the case that $L=M$, the weight function is identically $1$, which gives a bound
\begin{align*}
    \left\vert\int_{\mf{v}_Q} 1_{K_p(p^e)}(e^{X})v^Q_M(1,e^{X})|I^Q_p(X)|_p^{\frac12}dX\right\vert \leq C_p^I p^{-\frac{e\dim\Ind_M^G\mathcal{O}_\gamma}{2}}.
\end{align*}
Here $C_p^J$ is the $p$-adic valuation of  some constant depending only on $J_p^Q$. In fact, this polynomial may be globally defined: Note that $\gamma_p$ is just $\gamma\in M(\Q)$ embedded into $M(\Q_p)$ through the inclusion $\Q\to\Q_p$. Thus, if we let $\mf{m}(\Q)$ be the $\Q$-points of the Lie algebra of $M$, and let $X\in\mf{m}(\Q)$ such that $e^X$ represents the same orbit as $\gamma$, the Jacobson-Morozov theorem tells us that $X$ is part of an $\mf{sl}_2$-triple, and we have a full decomposition of $\mf{m}(\Q)$ into eigenspaces as in the local case. Picking a basis for the eigenspaces $\mf{m}_1$ and $\mf{m}_2$ globally, which is then also a basis in the local cases, we see from the definition in (\cite{Rao}, Lemma $2$) that the polynomial $I^Q_p$ is actually defined algebraically over $\Q$, and is thus global. This implies that the mentioned constant is a rational number, and hence one may pick $C_p^I=1$ for all but finitely many $p$. Importantly, those $p$ do not depend on $N$.

With these inequalities established, the bound on the entire non-archimedean part now follows: Let $S(N)$ be the collection of primes dividing $N$. By iterative use of the decomposition formula (\ref{eq:finegeom}), we get that
\begin{align*}
    J_M(1_{K(N)},\gamma_f) = \sum_{\underline{L}\in \mathcal{L}(M)^{|S(N)|}} d_M(\underline{L})\prod_{p\in S(N)} J^{L_p}_M((1_{K(N)})_{P_p},\gamma_p),
\end{align*}
where $\underline{L} =(L_p)_{p\in S(N)}$ runs over all such tuples of Levi subgroups, with $P_p=L_pN_p$ a certain parabolic subgroup of $G$ chosen as in (\cite{Arthur0}, §§$17$-$18$). Here $d_M(\underline{L})$ are constants which are non-zero only if at most $\dim \mf{a}_M$ many elements of $\underline{L}$ is not equal to $M$ (see \cite{MzM2}, Lemma $8.2$). This implies that the number of such $\underline{L}$ that actually contributes is bounded by $|S(N)|^{\dim \mf{a}_M}$. As $|S(N)|\leq 2\log N$, collecting the identities and bounds in this section amounts to the following proposition.
\begin{prop}\label{nonarchbound}
There exists constants $c,m>0$ independent of $N$ such that
    \begin{align}
    \left\vert J_M(1_{K(N)},\gamma_f)\right\vert \leq c N^{-\frac{\dim\Ind^G_M\mathcal{O}_\gamma}{2}}(1+\log N)^m.
\end{align}
\end{prop}

\noindent Later, we will isolate the orbital integral related to the trivial orbit $\lbrace 1\rbrace$ in $G$, while needing to bound the remaining terms. As there are only finitely many unipotent orbits $\mathcal{O}_\gamma\in \mathcal{U}_M(\Q)$ for every $M$, and only finitely many standard Levi subgroups $M$, one can make the bound above independent of $M$ and $\mathcal{O}_\gamma$ by picking the exponent to be half the minimal dimension of the non-trivial unipotent orbits. Hence we set
\begin{align}\label{k(G)}
    k(G) \coloneqq \frac12 \min_{(M,\mathcal{O}_\gamma)\neq (G,\lbrace 1\rbrace)} \dim\Ind_M^G \mathcal{O}_\gamma.
\end{align}

This means that for every $(M,\mathcal{O}_\gamma)\neq (G,\lbrace 1\rbrace)$ we have 
\begin{align}
    \left\vert J_M(1_{K(N)},\mathcal{O}_\gamma)\right\vert \leq c N^{-k(G)}(1+\log N)^m.
\end{align}

\noindent We immediately remark that since we allow trivial induction, i.e. from $G$ to $G$, the constant $k(G)$ is equal to half the dimension of the minimal (nontrivial) unipotent orbit in $G$. This constant is strictly positive, and is expressible in terms of root systems, see e.g. (\cite{CM}, §$4.3$).

Furthermore, if $G$ is Richardson, such that every unipotent orbit in $G$ is induced from the trivial orbit in some parabolic subgroup, then by the usual dimension formula for induced orbits we see that for $P=L_PV_P$ the Levi decomposition of $P$, we have
\begin{align}\label{k(G)uniprad}
    k(G) = \frac12 \min_{P\neq G} (\dim \lbrace 1\rbrace+2\dim V_P) = \min_{P\neq G}\dim V_P.
\end{align}

\noindent In particular, from this it is easy to see that for $G=\SL(n)$ and $G=\GL(n)$ over $\Q$, we have $k(G)=n-1$. Even if $G$ is not Richardson, the right-hand side above is still an upper bound of $k(G)$, with equality if and only if the minimal unipotent orbit is induced from the trivial orbit in some unipotent radical, equivalently from the trivial orbit in a unipotent radical of minimal dimension.

%% file: 6.proof.tex
\noindent We continue in the setup that $G$ is a connected reductive group with $G(\R)$ noncompact, and let $K_f\subset G(\A_f)$ be a neat open compact subgroup. Let $K(N)\subset K_f$ be as in (\ref{adelicprinccongsub}) for $N\in \N$, $N\geq 3$, and $K$ a maximal compact subgroup of $G(\R)^1$. We let $\Tilde{X}$ be the globally symmetric space associated to $G$, and $X(N)\coloneqq X(K(N))$ the adelic locally symmetric space associated to $K(N)$ in $G$ as defined in (\ref{adeliclocsym}). Write $d=\dim \Tilde{X}$. Let $(\tau,V_\tau)$ be a finite-dimensional complex representation of $G(\R)^1$. Let $h_t^{\tau,p}\in \mathcal{C}^q(G(\R)^1)$ be the trace of the associated heat kernel as in (\ref{trheat}), and the associated analytic torsion $T_{X(N)}(\tau)$ be as in (\ref{analytictorsion}).

We now prove our main theorem, giving improved asymptotic behaviour in the approximation of $L^2$-torsion by analytic torsion. This is a strengthening of Theorem $1.5$ of \cite{MzM3} under slightly different assumptions. In contrast to their theorem, in the proof of our theorem we need to handle the global coefficients $a^M(S_N,\mathcal{O}_\gamma)$ arising in the fine geometric expansion. This causes difficulty as no general bound is known for them. They depend on the primes dividing $N$, and so might grow, possibly rapidly, as $N$ goes to infinity. They are believed to grow logarithmically, as already established for $\GL(n)$ (see \cite{Matz}). We state a version of the conjecture here. Let $S_N$ be the set of primes dividing $N$ together with $\infty$.

\begin{conj}\label{globalcoeffconj}
    There exists constants $b,c>0$ such that for all $N$, $M$ and $\mathcal{O}_\gamma\in\mathcal{U}_M(\Q)$ we have
    \begin{align*}
        \left\vert a^M(S_N,\mathcal{O}_\gamma)\right\vert \leq c(1+\log N)^b.
    \end{align*}
\end{conj}

\noindent For the main theorem, we either have to assume the conjecture, or assume that $N$ only varies over a set of positive integers with all prime divisors lying in some finite fixed set of primes, rendering the global coefficients bounded by a constant. As non-standard notation, let us call such a set \textit{prime-fixed}. A good example of a prime-fixed set is $\lbrace m^k\mid k\in\N\rbrace$ given any $m\in\N$. It is interesting to compare this to the property $K_j\xrightarrow[S]{} 1$ defined in (\cite{MzM3}, $(1.13)$).

We will also need to assume that $G$ satisfies the two properties (TWN) and (BD), introduced in \cite{FLM}. They are named \textit{tempered winding number} and \textit{bounded degree} respectively, and are required to control the behaviour of intertwining operators showing up on the spectral side of the trace formula. These properties are conjectured to hold for all reductive groups over local fields and number fields, and have been established for $\SL(n)$ and $\GL(n)$ in (\cite{FLM}, Proposition $5.5$ and Theorem $5.15$), as well as for a large class of reductive groups, including quasi-split classical groups, in \cite{FL2} and \cite{FL3}.

Define $T^{(2)}_{X(N)}(\tau)$ to be the $L^2$-torsion of $X(N)$ as in \cite{Lott}. We can now state our theorem.

\begin{thm}\label{mytheoremgeneralized}
    Assume that $G$ is a connected reductive quasi-split group with $G(\R)$ noncompact. Let $\tau$ be an irreducible and $\lambda$-strongly acyclic representation of $G(\R)^1$, for some $\lambda>0$ only depending on $G$. Assume that $G$ satisfies the properties $(TWN)$ and $(BD)$ (see \cite{FLM}). Assume either Conjecture \ref{globalcoeffconj} or that $N$ varies over a prime-fixed set. Then there exists some $a>0$ such that
    \begin{align*}
        \log T_{X(N)}(\tau) = \log T^{(2)}_{X(N)}(\tau) + O\left(\vol(X(N)) N^{-k(G)}(\log N)^a\right)
    \end{align*}
    as $N$ tends to infinity.
\end{thm}

\noindent The rest of this section is a proof of this theorem.

\bigskip
\subsection{The first reduction}\label{reductionsection}

Assume the setup of the theorem. It follows from (\cite{MzM3}, Proposition $12.1$) that there exists $C>0$ such that for any $\epsilon>0$ and $t\geq 1$ we have
\begin{align}\label{speclargetbound}
    \left\vert J_{\spec}(h_t^{\tau,p}\otimes \chi(N))\right\vert\leq Ce^{-\lambda(1-\epsilon)t}\vol(X(N)).
\end{align}

\noindent This critically uses that $G$ satisfies (TWN) and (BD). The fact that their constant $c$ can be replaced by $\lambda(1-\epsilon)$ follows by closely examining the proof of (\cite{MzM2}, Lemma $6.3$), on which the proof of the above relies. Equipped with this we start the analysis of 
\begin{align}
    \FP_{s=0}\left(\frac{\zeta_p(s,\tau)}{s}\right) = \FP_{s=0}\left(\frac{1}{s\Gamma(s)}\int_0^\infty J_{\geo}(h_t^{\tau,p}\otimes \chi(N))t^{s-1}dt\right).
\end{align}

\noindent Note that $\Gamma(s)$ has a simple pole at $s=0$ with residue $1$, hence the fraction on the right hand side is holomorphic at $0$. Let $T\geq 1$. We split up the integral, and use that $\FP_{s=a}(f)=f(a)$ if $f$ is holomorphic at $s=a$, to express the above as

\begin{align}\label{splitintegral}
    \FP_{s=0}\left(\frac{1}{s\Gamma(s)}\int_0^T J_{\geo}(h_t^{\tau,p}\otimes \chi(N))t^{s-1}dt\right)+\int_T^\infty J_{\geo}(h_t^{\tau,p}\otimes \chi(N))t^{-1}dt.
\end{align}

\noindent By the Arthur-Selberg trace formula, any bound on $J_{\spec}$ also applies to $J_{\geo}$. In particular, applying (\ref{speclargetbound}) to the second term above yields the following lemma.

\begin{lem}\label{bigTintegralbound}
    There exists $C>0$ such that for any $\epsilon>0$, we have
    \begin{align*}
        \left\vert \int_T^\infty J_{\geo}(h_t^{\tau,p}\otimes \chi(N))t^{-1}dt\right\vert \leq Ce^{-\lambda(1-\epsilon)T}\vol(X(N)).
    \end{align*}
\end{lem}

\noindent That is the second term of (\ref{splitintegral}) handled, and we return to the first term. We focus on the integral. We want to replace the test function with its compactified version defined in (\ref{compactificationgeneral}). This is possible due to the following bound.

\begin{lem}[\cite{MzM3}, Proposition $12.3$]
    Let $R>0$. There exists constants $C_1,C_2,C_3>0$ such that
    \begin{align*}
        \left\vert J_{\spec}(h_t^{\tau,p}\otimes\chi(N))-J_{\spec}(h_{t,R}^{\tau,p}\otimes\chi(N))\right\vert \leq C_3 e^{-C_1 R^2/t+C_2t}.
    \end{align*}
\end{lem}

\noindent This is the second and final step which directly uses the properties (TWN) and (BD). The constants $C_1$ and $C_2$ may be chosen independently of $\tau$. Combining this lemma with Proposition \ref{compactunipidentitygeneral}, we get that for $N$ sufficiently large and $R\leq C_n \log N$,
\begin{align}\label{geoisunip+error}
    \int_0^T J_{\geo}(h_t^{\tau,p}\otimes \chi(N))t^{s-1}dt = \int_0^T J_{\unip}(h_{t,R}^{\tau,p}\otimes \chi(N))t^{s-1}dt + E_1(s,R,T),
\end{align}

\noindent where $E_1(s,R,T)$ is an error term given by an integral convergent for all $s\in\C$, and for $s=0$ bounded by
\begin{align}\label{error1general}
    |E_1(0,R,T)|\leq C_3 e^{-C_4 R^2/T+C_2 T}\int_0^{T/R^2}e^{-C_4/t}t^{-1}dt \:\vol(X(N)).
\end{align}
for some $C_4$ only depending on $G$. 

\bigskip
\subsection{Separating the leading term}

Having dealt with the latter term of the right hand side of (\ref{geoisunip+error}), we use the fine geometric expansion of the trace formula to analyze the former. By Theorem \ref{generalorbdistributionintro}, we may write $J_{\unip}$ as a sum over pairs $(M,\mathcal{O}_\gamma)$ of orbital integrals. Let $J_{\unip-1}$ denote $J_{\unip}$ subtracted the term associated to $(M,\mathcal{O}_\gamma)=(G,\lbrace 1\rbrace)$. One can check that this term simplifies to $h_{t,R}^{\tau,p}(1)$. Thus, We may write 
\begin{align}\label{seperateleadingterm}
    \int_0^T J_{\unip}(h_{t,R}^{\tau,p}\otimes \chi(N))t^{s-1}dt &= \int_0^T h_{t,R}^{\tau,p}(1)t^{s-1}dt \vol(X(N)) \\
    &+\int_0^T J_{\unip-1}(h_{t,R}^{\tau,p}\otimes \chi(N))t^{s-1}dt
\end{align}

\noindent Note that $h_{t,R}^{\tau,p}(1)= h_t^{\tau,p}(1)$, i.e. the compactification does not matter when evaluating at $1$. By (\cite{MP}, $(5.11)$) this has a asymptotic expansion for small $t$
\begin{align}\label{heatsmalltasympgeneral}
    h^{\tau,p}_t(1) \sim \sum_{i=0}^\infty a_i t^{i-d/2}
\end{align}
as $t\to 0$. Furthermore, as a special case of Proposition \ref{heatkerlarget}(a), we have a large $t$ bound for some $C>0$ and $t\geq 1$,
\begin{align*}
    |h^{\tau,p}_t(1)|\leq Ce^{-\lambda t}.
\end{align*}
Together, this implies the absolute convergence in the complex half-plane $\Re(s) > \frac{d}{2}$ and the existence of a meromorphic continuation of its Mellin transform. Furthermore, the upper bound allows us to bound the error term arising in the expression
\begin{align}\label{heatfullmellin}
    \int_0^T h_{t,R}^{\tau,p}(1)t^{s-1}dt = \int_0^\infty h_t^{\tau,p}(1)t^{s-1}dt + E_2(s,T).
\end{align}
Here $E_2(s,T)$ is holomorphic in $s$ and given by an integral which is absolutely convergent everywhere satisfying
\begin{align}\label{error2general}
    |E_2(0,T)|\leq C e^{-\lambda T}
\end{align}
for $t\geq 1$ and some constant $C>0$. 

\subsection{Handling the non-leading terms}

We turn our attention to the truncated Mellin transform of $J_{\unip-1}(h^{\tau,p}_{t,R}\otimes \chi(N))$ from (\ref{seperateleadingterm}). We let $S=S(N)\sqcup \lbrace\infty\rbrace$. Using the fine geometric expansion, we write
\begin{align}
    J_{\unip-1}(h^{\tau,p}_{t,R}\otimes \chi(N)) = \sum_{(M,\mathcal{O}_\gamma)\neq (G,\lbrace 1\rbrace)}a^M(S,\mathcal{O}_\gamma)J_M(h^{\tau,p}_{t,R}\otimes \chi(N),\mathcal{O}_\gamma).
\end{align}

\noindent By (\ref{eq:finegeom}), each of the orbital integrals can be decomposed into local constituents
\begin{align}
    J_M(h^{\tau,p}_{t,R}\otimes \chi(N),\mathcal{O}_\gamma) = \sum_{L_1,L_2\in\mathcal{L}(M)}d_M(L_1,L_2)J^{L_1}_M(h^{\tau,p}_{t,R},\mathcal{O}_{\gamma_\infty})J^{L_2}_M(\chi(N),\mathcal{O}_{{\gamma_f}}).
\end{align}
Since only the archimedean orbital integral sees the variable $t$, we may write the truncated Mellin transform of $J_M(h^{\tau,p}_{t,R}\otimes \chi(N),\mathcal{O}_\gamma)$ as
\begin{align*}
    \sum_{L_1,L_2\in\mathcal{L}(M)}d_M(L_1,L_2)\left(\int_0^TJ^{L_1}_M(h^{\tau,p}_{t,R},\mathcal{O}_{\gamma_\infty})t^{s-1}dt\right)J^{L_2}_M(\chi(N),\mathcal{O}_{\gamma_f}).
\end{align*}

\noindent The local orbital integrals were analyzed in the previous section. Also, an asymptotic expansion of $J^{L_1}_M(h^{\tau,p}_{t,R},\mathcal{O}_{\gamma_\infty})$ as $t\to 0$ similar to (\ref{heatsmalltasympgeneral}) is obtained by combining Proposition $7.2$ and $(11.21)$ of \cite{MzM3}. By Proposition \ref{archbound}, when $\lambda >c'$, we then see that the truncated Mellin transform $\int_0^TJ^{L_1}_M(h^{\tau,p}_{t,R},\mathcal{O}_{\gamma_\infty})t^{s-1}dt$ enjoys the same properties as the one of $h_t^{\tau,p}(1)$. In particular, not only is it absolutely convergent in some half-plane and has a meromorphic continuation to all of $\C$, but it is also equal to the full Mellin transform (cf. (\ref{heatfullmellin})) up to some error term $E_3(s,T)$ satisfying
\begin{align}
    |E_3(0,T)|\leq C e^{-(\lambda-c') T}
\end{align}
for some constant $C>0$. Then we see that 
\begin{align*}
    \FP_{s=0}\left(\frac{1}{s\Gamma(s)}\int_0^TJ^{L_1}_M(h^{\tau,p}_{t,R},\mathcal{O}_{\gamma_\infty})t^{s-1}dt\right) = \FP_{s=0}&\left(\frac{1}{s\Gamma(s)}\int_0^\infty J^{L_1}_M(h^{\tau,p}_{t,R},\mathcal{O}_{\gamma_\infty})t^{s-1}dt\right) \\
    &\qquad\qquad\qquad\qquad\qquad+ E_3(0,T)
\end{align*}
may be bounded by a uniform constant as $T$ grows to infinity.

The non-archimedean orbital integrals were appropriately bounded in Proposition \ref{nonarchbound}. Combining these bounds with the decompositions in the current section, and making the technical assumption previously discussed to handle global coefficients, we have shown the following.

\begin{prop}\label{unip-1boundgeneral}
    Assume Conjecture \ref{globalcoeffconj} or that $N$ varies over a prime-fixed set, and assume $\lambda>c'$. Then for $N$ sufficiently large and $R\leq C_n\log N$, there exists constants $C,b>0$ with $b$ only depending on $G$ such that
    \begin{align}
        \left\vert\FP_{s=0}\left(\frac{1}{s\Gamma(s)}\int_0^T J_{\unip-1}(h^{\tau,p}_{t,R}\otimes \chi(N)) t^{s-1}dt\right)\right\vert \leq C N^{-k(G)}(\log N)^b \vol(X(N)).
    \end{align}
\end{prop}

\subsection{Approximating $L^2$-torsion}

Collecting the results of Section \ref{reductionsection}, we may insert into the definition of analytic torsion and see
\begin{align*}
    \log T_{X(N)}(\tau) &= \FP_{s=0}\left(\frac12\frac{1}{s\Gamma(s)}\int_0^\infty\sum_{p=1}^d (-1)^p p \, h^{\tau,p}_t(1)t^{s-1}dt\right) \vol(X(N)) \\
    &+ \FP_{s=0}\left(\frac12\sum_{p=1}^d (-1)^p p\,\frac{1}{s\Gamma(s)} \int_0^T J_{\unip-1}(h^{\tau,p}_{t,R}\otimes\chi(N))t^{s-1}dt\right) \\
    &+ \frac12\sum_{p=1}^d (-1)^p p\, \left(\int_T^\infty J_{\geo}(h_t^{\tau,p}\otimes\chi(N))t^{-1}dt + E_1(0,R,T)+E_2(0,T)\right)
\end{align*}

\noindent The first term on the right hand side is the $L^2$-torsion, as we show below. The second term was just handled in the previous subsection, Proposition \ref{unip-1boundgeneral}. The error terms, all contained in the third term above, will be dealt with in a moment.

We recall from \cite{Lott} that $T^{(2)}_{X(N)}(\tau)=t^{(2)}_{\Tilde{X}}(\tau)\vol(X(N))$ for $t^{(2)}_{\Tilde{X}}(\tau)$ a computable constant depending only on $\tau$ and the universal covering $\Tilde{X}$, given as
\begin{align}\label{torsionconstant}
    t^{(2)}_{\Tilde{X}}(\tau)\coloneqq \frac12 \frac{d}{ds}\left(\frac12\frac{1}{\Gamma(s)}\int_0^\infty\sum_{p=1}^d (-1)^p p \, h^{\tau,p}_t(1)t^{s-1}dt\right)\Bigg\vert_{s=0} .
\end{align}

\noindent This is equal to the first term above as the meromorphic continuation of 
$$\frac{1}{\Gamma(s)}\int_0^\infty\sum_{p=1}^d (-1)^p p \, h^{\tau,p}_t(1)t^{s-1}dt$$
is holomorphic at $s=0$ (see \cite{BV}, §$4.4$). 

By Proposition \ref{unip-1boundgeneral}, the second term above is of the form
\begin{align*}
    O\left(N^{-k(G)}(\log N)^b \vol(X(N))\right)
\end{align*}
for $N$ large enough and $b>0$ some constant. This is the correct size for the asymptotic in Theorem \ref{mytheoremgeneralized}. Hence, the only missing step in proving the theorem is to establish the same bound for the error terms.

\bigskip
\subsection{Correct asymptotics of error terms}

In this final section, we critically use that we may choose $\lambda$, $T$, and $R$ freely (within the constraints $\lambda >0$, $T\geq 1$, $0<R\leq C_n\log N$), and that no essential constants depend on their choice. To ensure that the error term $E_2$ vanishes sufficiently quickly, we need to pick $T\sim \log N$. Thus, let $T=\beta \log N$, with the constant $\beta>0$ to be determined in a moment. We also pick the maximal allowed value for $R$, namely $R=C_n\log N$. These choices mean that our error terms take the form
\begin{align*}
    \int_T^\infty J_{\geo}(h_t^{\tau,p}\otimes\chi(N))t^{-1}dt &= O\left(N^{-\lambda(1-\epsilon)\beta}\vol(X(N))\right) \\
    E_1(0,C_n\log N,\beta\log N) &= O\left((N^{-C_4C_n^2/\beta+C_2\beta}\vol(X(N))\right)\\
    E_2(0,\beta\log N) &= O\left(N^{-\lambda \beta}\vol(X(N))\right).
\end{align*}
This follows from Lemma \ref{bigTintegralbound}, and equations (\ref{error1general}) and (\ref{error2general}), respectively. Recall that $C_2,C_4,C_n>0$ are all fixed constants, only depending on $G$. 

Choose $\beta$ small such that $-C_4C_n^2/\beta+C_2\beta \leq -k(G)$. This ensures the correct asymptotic for $E_1$. Now, pick the spectral gap $\lambda$ large such that $-\lambda(1-\epsilon) \beta \leq -k(G)$ and $\lambda >c'$ as well. Note that this choice is only dependent on $G$. This ensures that the remaining error terms have the correct asymptotic. We have now finished the proof of Theorem \ref{mytheoremgeneralized}.

\clearpage

%% file: 7.SL2F.tex
\noindent In this section, we will present some of the applications of Theorem \ref{mytheoremgeneralized} in the setting of $\SL(2)$ defined over a number field $F$ of degree $n$ over $\Q$, where a Cheeger-Müller formula has been established in \cite{MR3}. The setup is the following.

Let $G_0=\SL(2)/F$, and set $G=\text{Res}_{F/\Q}(G_0)$ its restriction of scalars to $\Q$. Then $G(\Q)=\SL(2,F)$, and we can continue working over $\Q$. Let $r_1$ and $r_2$ be respectively the number of real, respectively pairs of complex, embeddings of $F$, such that
\begin{align*}
    G(\R) = \SL(2,\R)^{r_1}\times \SL(2,\C)^{r_2}.
\end{align*}

\noindent Let $K=\SO(2)^{r_1}\times \SU(2)^{r_2}$. Then the associated symmetric space $\Tilde{X}\coloneqq G(\R)/K$ is

\begin{align*}
    \Tilde{X} = (\mathbb{H}^2)^{r_1}\times (\mathbb{H}^3)^{r_2}.
\end{align*}

\noindent We embed $\SL(2,\mathcal{O}_F)$ as a discrete subgroup of $G(\R)$ induced by the embedding of $F$ in $\R^{r_1}\times \C^{r_2}$. Now let $\Gamma\subset\SL(2,\mathcal{O}_F)$ be a torsion free finite index subgroup, also discretely embedded through the inclusion. Then $X\coloneqq \Gamma\backslash \Tilde{X}$ is a (noncompact) locally symmetric manifold of finite volume. With the parametrizations
\begin{align*}
    \mathbb{H}^2 = \lbrace x+iy\mid x\in \R, y>0\rbrace, \qquad\mathbb{H}^3=\lbrace z+is\mid z\in\C, s>0\rbrace,
\end{align*}
we give an invariant metric on $\Tilde{X}$ by
\begin{align*}
    \Tilde{g} = \sum^{r_1}_{i=1}\frac{dx_i^2+dy_i^2}{y_i^2}+2\sum^{r_2}_{j=1}\frac{|dz_j|^2+ds_j^2}{s_j^2}.
\end{align*}

\noindent We let $g$ denote the metric on $X$ induced by $\Tilde{g}$. Let $(\tau,V_\tau)$ be a finite dimensional irreducible representation of $G(\R)$, and let $E$ be the flat vector bundle associated to $\tau|_\Gamma$. As in Section \ref{section:preliminaries}, one may equip this bundle with a canonical bundle metric $h$ defined by an admissible inner product on $V_\tau$. Let $T(X,E,g,h)$ denote the associated analytic torsion as defined in \cite{ARS}.

$X$ admits a natural compactification as a manifold with boundary $\overline{X}$ (see \cite{MR3}, §$3$). The Cheeger-Müller formula in this setting is the following. Let $\tau(\overline{X},E,\mu_X)$ denote the Reidemeister torsion of $(\overline{X},E)$ with respect to $\mu_X$, a particular basis of $H^*(\overline{X},E)$ as chosen in (\cite{MR3}, $(5.2)$).

\begin{thm}[\cite{MR3}, Theorem $1.1$]\label{cheegermueller}
    In the notation above, assume that $r_2$ is odd, $r_1>0$ and $r_1+r_2>2$. Then
    \begin{align*}
        T(X,E,g,h) = \tau(\overline{X},E,\mu_X).
    \end{align*}
\end{thm}

\noindent Their full result also allows for $r_1=0$ in exchange for further assumptions on the representation. 

Take a faithful $\Q$-rational representation $\rho:G\to\GL(V)$ and a lattice $\Lambda\subset V$ stabilized by $\SL(2,\mathcal{O}_F)$. Let $K_f\subset G(\A_f)$ denote the stabilizer of $\hat{\Z}\otimes \Lambda$ and subgroups $K(N)$ as defined in (\ref{adelicprinccongsub}). Set 
\begin{align*}
    \Gamma(N)\coloneqq G(\Q)\cap (G(\R)\times K(N)).
\end{align*}
If $\rho$ and $\Lambda$ are chosen correctly, these are exactly the principal congruence subgroups of level $N$ in $\SL(2,\mathcal{O}_F)$. As we are working with $\SL(2)$, strong approximation guarantees us that $X(N) = \Gamma(N)\backslash \Tilde{X}$.

The definition of analytic torsion given in \cite{ARS} matches the one given in Section \ref{section:analytictorsion} in this setup (see the discussion in \cite{MR3}, §$7$). Let $\rho_\R$ be the representation of $G(\R)$ induced by $\rho$. From now on, we assume that it decomposes into a sum of $\lambda$-strongly acyclic representations. One can now combine this with our Theorem \ref{mytheoremgeneralized} to achieve new asymptotics on cohomology.

Let $L_\rho$ be the local system of free $\Z$-modules over $X(N)$ associated to $\Lambda$. Assume that $E_{\rho_\R}$ is acyclic, such that $H^*(X(N),L_{\rho})$ is purely torsion. In particular there is no need for a choice of basis for the free part, and we omit it from the notation. Then by (\cite{Cheeger}, $(1.4)$) we have
\begin{align}\label{reidemeistercohom}
    \tau(\overline{X(N)}, E_{\rho_\R})^2 = \prod_{q=0}^d|H^q(\overline{X(N)},L_\rho)|^{(-1)^{q+1}}.
\end{align}

\begin{rmk}
    As we wish to apply Theorem \ref{mytheoremgeneralized}, it would be beneficial to compute $k(G)$ explicitly. As $\SL(n)$ is Richardson, we may use the description of $k(G)$ given in (\ref{k(G)uniprad}) as the minimal dimension of non-trivial unipotent radicals. As $G(\Q)=\SL(2,F)$, we clearly have a minimal non-trivial unipotent radical of dimension $1$ over $F$, namely $\lbrace\left( \begin{smallmatrix}1& \ast \\0 & 1\end{smallmatrix}\right)\rbrace$, which then has dimension $n = [F:\Q]$ over $\Q$. Thus, $k(G)=n$.
\end{rmk}

\begin{rmk}
Let us consider what $\lambda$-strongly acyclic representations look like for $G(\R)$. Assume for a moment that $\rho_\R$ is irreducible, such that it is a tensor product of representations on $\SL(2,\R)$ and $\SL(2,\C)$. The proof of (\cite{berland1}, Proposition $3.2$) shows that an irreducible representation, thought of as a dominant integral weight in the weight space, is $\lambda$-strongly acyclic if it is "far enough away" from the subspace of $\theta$-fixed weights, with $\theta$ being the Cartan involution. Considering $\SL(2,\C)$ as a real group, its irreducible representations are $\text{Sym}^a(V)\otimes \text{Sym}^b(\overline{V})$ for $V$ the standard representation and $\overline{V}$ its complex conjugate, and this is $\theta$-fixed if and only if $a=b$. One gets that this representation is $\lambda$-strongly acyclic for a given $\lambda>0$ if both $\max(a,b)$ and $|a-b|$ are sufficiently large. Furthermore, it is sufficient for the full tensor representation of $G(\R)$ to be $\lambda$-strongly acyclic that one of its components is $\lambda$-strongly acyclic.
\end{rmk}

\noindent With the remarks out of the way, we are ready to present the main theorem of the section.

\begin{thm}\label{theorem:SL2cohomology}
    Let $F$ be a number field of degree $[F:\Q] = n=r_1+2r_2$. Assume that $r_2$ is odd, $r_1>0$ and $r_1+r_2>2$. Assume further that $\rho_\R$ decomposes into a sum of $\lambda$-strongly acyclic representations, with $\lambda>0$ chosen as in Theorem \ref{mytheoremgeneralized}, and that $E$ is acyclic. Then there exists a contant $a>0$ such that
    \begin{align*}
        \frac12\sum_{q=0}^d(-1)^{q+1}\log|H^q(\overline{X(N)},L_\rho)| =  \log T^{(2)}_{X(N)}(\rho_\R) + O(\vol(X(N))N^{-n} \log (N)^a)
    \end{align*}
    as $N$ tends to infinity.
\end{thm}

\begin{proof}
    If one assumes Conjecture \ref{globalcoeffconj}, this is simply the application of Theorem \ref{mytheoremgeneralized} to Theorem \ref{cheegermueller}, phrased in terms of cohomology using (\ref{reidemeistercohom}). Indeed, note that the properties (TWN) and (BD) are established in this setting as discussed above Theorem \ref{mytheoremgeneralized}. The conjecture is proven for $\GL(n)$ over any number field in \cite{Matz}, and one may relate the trace formula of $\SL(n)$ to the one of $\GL(n)$ as in (\cite{MzM2}, §$11$). Then, applying Theorem \ref{mytheoremgeneralized} to $\text{Res}_{F/\Q}(\GL(n)/F)$, we get the unconditional result for $G=\text{Res}_{F/\Q}(\SL(n)/F)$. 
\end{proof}

\noindent Note that we do not cover the case of $r_1=0$, $r_2=1$, i.e. $G(\R)=\SL(2,\C)$. However, this is accomplished by a different route in Corollary \ref{cor:SO(3,1)}. The result above gives second order terms to the asymptotic growth of cohomology when the deficiency is $1$, i.e. when $r_2=1$ (cf. \cite{MR3}, Theorem $1.3$). However, the main motivation for seeking the asymptotics of Theorem \ref{mytheoremgeneralized} was in its applications when the deficiency is strictly greater than $1$. We spell it out here for emphasis.

\begin{cor}\label{deficiencyabove1cor}
    Assume further that $r_2>1$, such that $\delta(G)>1$. Then we have that $\log T^{(2)}_{X(N)}(\rho_\R)=0$, and hence,
    \begin{align*}
        \sum_{q=0}^d(-1)^{q+1}\frac{\log|H^q(\overline{X(N)},L_\rho)|}{[\Gamma(1):\Gamma(N)]} = O\left(\frac{(\log N)^a}{N^n}\right) ,
    \end{align*}
    as $N$ tends to infinity.
\end{cor}

\begin{rmk}
    It is possible that there is significant cancellation happening in the alternating sum, and that this upper bound is not satisfied for each individual cohomology group. Heuristically, however, this seems not to be the case. The analogous asymptotic for individual homology groups has been established for degrees $i\leq d-2$ of integral homology of principal congruence subgroups of $\SL(d,\Z)$ by Abert--Bergeron--Fraczyk--Gaboriau, see (\cite{ABFG}, Theorem B). Here, they attain an upper bound of $O\left(\frac{\log N}{N^{d-1}}\right)$, and $d-1=k(\SL(d))$ when $\SL(d)$ is defined over $\Q$. This suggests that one should not expect significant cancellation in the alternating sum above.
\end{rmk}

%% file: 8.SOn1.tex
\noindent Let $n\in\N$ be an odd integer, $n=2d+1$, and fix $a_1,\dots,a_n\in\N$. Let $G$ be the algebraic group over $\Q$ defined by the quadratic form
\begin{align*}
    q(x_1,\dots,x_{n+1}) = a_1x_1^2+\dots+a_nx_n^2-x_{n+1}^2,
\end{align*}
such that for any $\Q$-algebra $A$, we have that
\begin{align}\label{SO(n,1)functor}
    G(A) = \lbrace g\in\SL(n+1,A)\mid q(gx)=q(x),\:x\in A^{n+1}\rbrace.
\end{align}
Note that this also gives us a natural $\Z$-structure. We have $G(\R)\cong \SO(n,1)$, and we pick $K$ a maximal compact subgroup, which is then isomorphic to $\SO(n)$. We present here some of the work done in \cite{MR1}, \cite{MR2} required in applying our results on analytic torsion to cohomology in this setting. Critically, this includes the Cheeger-Müller type formula of the former.

We let $\Gamma=G(\Z)$, For $N\in\N$ define the principal congruence subgroup
\begin{align*}
    \Gamma(N)\coloneqq \ker(G(\Z)\to G(\Z/N\Z)).
\end{align*}
If $N\geq 3$, this subgroup is neat, in particular torsion-free. We assume $N\geq 3$ from now on. We consider the associated locally symmetric manifold
\begin{align*}
    X(N)=\Gamma(N)\backslash G(\R)/K,
\end{align*}
which is then a hyperbolic manifold of finite volume, and we denote its hyperbolic metric by $g_{X(N)}$.

\begin{rmk}
    Given the explicit description in (\ref{SO(n,1)functor}), it is clear that this $\Gamma(N)$ coincides with the subgroup of level $N$ coming from the adelic setup (see (\ref{arithmsubsfromadelic}) and (\ref{adelicprinccongsub})). In particular, we may directly import the results of Part A, most importantly the estimates on analytic torsion in Theorem \ref{mytheoremgeneralized}.
\end{rmk}

\noindent Let $(\rho,V)$ be a finite dimensional $\Q$-rational representation of $G$, with a lattice $\Lambda\subset V$ invariant under $\Gamma$. Let $L_\rho$ be the local system of free $\Z$-modules over $X(N)$ associated to $\Lambda$. Assume that the associated representation $\rho_\R$ of $G(\R)$ is irreducible, with associated flat vector bundle $E_{\rho_\R}$ on $X(N)$. Let $h_{E_{\rho_\R}}$ denote its induced hermitian metric.

We will write $\overline{X(N)}$ for the Borel-Serre compactification of $X(N)$, and by abuse of notation also use $E_{\rho_\R}$ to denote the induced flat vector bundle on the compactification. Let $\tau(\overline{X(N)};E_{\rho_\R},\mu_{X(N)})$ denote the Reidemeister torsion of $\overline{X(N)}$ with respect to the basis $\mu_{X(N)}$ of $H^*(\overline{X(N)},E_{\rho_\R})$ given by Eisenstein series (see \cite{MR1}), and similarly $\tau(\partial \overline{X(N)},E_{\rho_\R},\mu_{\partial X(X)})$ for the Reidemeister torsion associated to the boundary.

\begin{prop}[\cite{MR1}, Theorem $1.1$]\label{cheegermüllerSO(n,1)}
    In the above setting, we have
    \begin{align*}
        \log T(X(N);E_{\rho_\R},g_X,h_E) &= \log \tau(\overline{X(N)},E_{\rho_\R},\mu_{X(N)}) \\
        &- \frac12\tau(\partial \overline{X(N)},E_{\rho_\R},\mu_{\partial X(X)})+\kappa_{\rho_\R}(X(N))c_\rho,
    \end{align*}
    with $\kappa_{\rho_\R}(X(N))$ the number of cusps of $X(N)$ contributing to cohomology and $c_\rho$ some constant depending only on $\rho$.
\end{prop}

\noindent This is the Cheeger-Müller type formula that we will need. The result cited is more general than presented here, in particular it covers a much larger class of arithmetic subgroups. As a next step, we will need to deal with the latter terms on the right-hand side above. Write
\begin{align*}
    \overline{\rho}_\R\coloneqq\rho_\R\oplus\rho_\R^*,\quad \overline{E}_{\rho_\R}\coloneqq E_{\rho_\R}\oplus E_{\rho_\R}^*,
\end{align*}
where $\rho_\R^*$ is the contragredient representation of $\rho_\R$, and $E_{\rho_{\R}}^*$ is its associated bundle. Set also 
\begin{align*}
    t^{(2)}_{\Tilde{X}}(\overline{\rho}_\R) = t^{(2)}_{\Tilde{X}}(\rho_\R)+t^{(2)}_{\Tilde{X}}(\rho_\R^*),
\end{align*}
(recall the definition in (\ref{torsionconstant})), and $T_{X(N)}(\overline{\rho}_\R)$ similarly. Let $\overline{\mu}_{X(N)}$ be the analogous basis to $\mu_{X(N)}$ when the vector bundle is $\overline{E}_{\rho_{\R}}$, and $\overline{\mu_{\partial X(N)}}$ for the boundary. One of the main points of this setup is that these objects are self-dual, allowing for duality arguments, the principal being that the boundary contribution is trivial, i.e.

\begin{align}\label{boundaryis0}
    \tau(\partial \overline{X(N)},\overline{E}_{\rho_\R},\overline{\mu_{\partial X(X)}}) = 0,
\end{align}

\noindent see (\cite{MR2}, ($4.13$)). Thus, we only need to deal with the contribution coming from the number of cusps. For this we need a bit of notation, and to mention some of the properties of our congruence subgroups.

\begin{defn}
    We say that a parabolic subgroup $P\subset G(\R)$ with unipotent radical $U_P$ is $\Gamma(N)$-cuspidal if $\Gamma(N)\cap U_P$ is a lattice in $U_P$.
\end{defn}

\noindent We let
\begin{align*}
    \mf{P}_{\Gamma(N)} = \lbrace P_{N,1},\dots,P_{N,\kappa(\Gamma(N))}\rbrace
\end{align*} 
be the set of $\Gamma(N)$-conjugacy classes of $\Gamma(N)$-parabolic subgroups of $G(\R)$. We set
\begin{align*}
    \kappa(X(N)) = \kappa(\Gamma(N)) = |\mf{P}_{\Gamma(N)}|,
\end{align*}
and note that this is always finite. It is clear that we have
\begin{align*}
    \kappa_{\rho_\R}(X(N)) \leq \kappa(X(N)).
\end{align*}
The key to controlling the contribution from the cusps is the following.

\begin{prop}\label{cuspcount}
    Let $\Gamma_N$ be a sequence of finite index subgroups of $\Gamma$ satisfying the following assumptions:
    \begin{enumerate}
        \item[(a)] $[\Gamma:\Gamma_N]\to \infty$ as $N\to\infty$,
        \item[(b)] Every non-identity element of $\Gamma$ occurs in $\Gamma_N$ for only finitely many $N$,
        \item[(c)] for $N$ large enough, any $\Gamma_N$-cuspidal parabolic subgroup $P$ with Levi decomposition $P=L_PU_P$ satisfy
        \begin{align*}
            \Gamma_N\cap P = \Gamma_N\cap U_P,
        \end{align*}
    \end{enumerate}
    Then we have that
    \begin{align}\label{cuspestimate}
        \frac{\kappa(\Gamma_N)}{[\Gamma:\Gamma_N]} \leq \sum_{P_l\in\mf{P}_\Gamma}\frac{1}{[\Gamma\cap U_{P_l}:\Gamma_N\cap U_{P_l}]}.
    \end{align}
\end{prop}

\begin{proof}
    This follows directly from the proof of (\cite{MP2}, Proposition $8.6$).
\end{proof}

\noindent All these properties are satisfied for $\Gamma_N=\Gamma(N)$ and are easily checked. Importantly, we can also estimate the indices appearing on the right-hand side of (\ref{cuspestimate}): $U_P$ is homeomorphic to $\R^{\dim U_P}$ as a topological space, and under this homeomorphism we have that $U_P\cap \Gamma$ corresponds to the lattice $ \Z^{\dim U_P}$ with $U_P\cap \Gamma(N)$ corresponding to the sublattice $(N\Z)^{\dim U_P}$. In particular, the quotient $(U_P\cap \Gamma)/(U_P\cap \Gamma(N))$ corresponds to $(\Z/N\Z)^{\dim U_P}$, giving us a precise description of the index as the cardinality of the quotient,
\begin{align*}
    [U_P\cap \Gamma:U_P\cap \Gamma(N)] = N^{\dim U_P}.
\end{align*}
These types of terms also showed up in our analysis of analytic torsion. Indeed, recall that $k(G)$, defined in (\ref{k(G)}), is bounded from above by the minimal dimension of unipotent orbits. Combining this with Proposition \ref{cuspcount} and our analysis above, we get the following.

\begin{cor}\label{cuspbound}
    We have that
    \begin{align*}
        \frac{\kappa_{\rho_\R}(X(N))}{[\Gamma:\Gamma(N)]} = O\left(\frac{1}{N^{k(G)}}\right).
    \end{align*}
\end{cor}

\noindent As the constant $k(G)$ is central in the error terms we get below, let us compute it explicitly. We may identify the complexified Lie algebra of $G$ with
\begin{align*}
    \mf{g}_\C \cong \mf{so}(n+1,\C),
\end{align*}
which has root system $D_{d+1}$ (recall $n=2d+1$). By (Theorem $1$, \cite{Wang}), the minimal nontrivial nilpotent orbit $\mathcal{O}_{min}$ in a complex simple Lie algebra $\mf{l}$ is $2h^\vee-2$, where $h^\vee$ is the dual Coxeter number of $\mf{l}$. For $n\geq 5$, $\mf{so}(n+1,\C)$ is simple, and as $D_{d+1}$ has dual Coxeter number $2d$ (see e.g. \cite{Kac}, p. $80$), we see that the minimal nontrivial orbit must have dimension 
\begin{align*}
    \dim\mathcal{O}_{min} = 4d-2 = 2n-4.
\end{align*}
This is a priori the (complex) dimension of the minimal nilpotent orbit in the complexified Lie algebra, but it extends directly to over $\Q$, e.g. as the classification of orbits is given by Jordan normal forms which for unipotent elements are defined over $\Q$. Then, by the discussion following Definition \ref{k(G)}, we see that 
\begin{align*}
    k(G) = \frac12 \dim\mathcal{O}_{min} = n-2, \quad n\geq 5.
\end{align*}

\noindent In the case of $n=3$, we have the exceptional isomorphism $\mf{so}(4) = \mf{so}(3)\oplus \mf{so}(3)$, showing that the Lie algebra is not simple, and the formula of Wang is not applicable. Luckily, this also corresponds to the exceptional isomorphism $\text{Spin}(3,1)\cong \SL(2,\C)$, and as the former is the double cover of $\SO(3,1)$, they have the same dimension of orbits. We know that, considered as a real group, the dimension of the minimal unipotent orbit of $\SL(2,\C)$ is $4$ (this is easy to compute using that it is Richardson), hence the same holds for $\SO(3,1)$, and thus $k(\SO(3,1)) = 2$ by the same formula as above. We summarize:
\begin{lem}
    For $n\geq 3$ odd, we have that
    \begin{align*}
    k(\SO(n,1)) = \begin{cases}
        2 & \text{if } n = 3 \\
        n-2 & \text{otherwise}.
    \end{cases}
\end{align*}
\end{lem}

\noindent Combining now our main theorem on analytic torsion (Theorem \ref{mytheoremgeneralized}) with the Cheeger-Müller formula (Proposition \ref{cheegermüllerSO(n,1)}), and using our results on the secondary contributions (\ref{boundaryis0}) and Corollary \ref{cuspbound}, we get the following result.

\clearpage

\begin{prop}\label{reidemeisterasympSO(n,1)}
    Assume that $\rho_\R$ is a $\lambda$-strongly acyclic representation of $G(\R)$. Assume $G$ satisfies (TWN) and (BD), and assume either Conjecture \ref{globalcoeffconj} or that $N$ varies over a prime-fixed set. Then there exists $a>0$ such that
    \begin{align*}
        \frac{\log \tau(\overline{X(N)},\overline{E}_{\rho_\R},\overline{\mu}_{X(N)})}{[\Gamma:\Gamma(N)]} = T^{(2)}_{X(1)}(\overline{\rho}_\R)+O\left(\frac{(\log N)^a}{N^{k(G)}}\right),
    \end{align*}
    as $N$ goes to infinity.
\end{prop}

\noindent All that is left is to express Reidemeister torsion in terms of torsion cohomology. Assume that $L_\rho$ is acyclic, such that the associated cohomology is purely torsion. Then we no longer need to consider a basis of the free cohomology, and so will omit this from the notation going forward. We let $\overline{L}_\rho = L_\rho\oplus L_{\rho^*}$. Again by (\cite{Cheeger}, $(1.4)$) we have
\begin{align}\label{reidemeistercohomSO(n,1)}
    \tau(\overline{X(N)}, \overline{E}_{\rho_\R})^2 = \prod_{q=0}^d|H^q(\overline{X(N)},\overline{L}_\rho)|^{(-1)^{q+1}}.
\end{align}

\noindent By combining Proposition \ref{reidemeisterasympSO(n,1)} and (\ref{reidemeistercohomSO(n,1)}), we get improved asymptotics in the growth of torsion.

\begin{thm}
    Assume that $\rho_\R$ is a $\lambda$-strongly acyclic representation of $G(\R)$, and that the vector bundle $L_\rho\otimes \C$ is acyclic. Assume $G$ satisfies (TWN) and (BD), and assume either Conjecture \ref{globalcoeffconj} or that $N$ varies over a prime-fixed set. Then there exists $a>0$ such that
    \begin{align*}
        \frac12\sum_{q=0}^d(-1)^{q+1}\frac{\log|H^q(\overline{X(N)},\overline{L}_\rho)|}{[\Gamma:\Gamma(N)]} = T^{(2)}_{X(1)}(\overline{\rho}_\R)+O\left(\frac{(\log N)^a}{N^{k(G)}}\right).
    \end{align*}
\end{thm}

\noindent With only a little work, one can make an even stronger statement in the case of $n=3$. Since $\SO(p,q)$ is quasi-split iff $|p-q|\leq 2$, we see that $\SO(3,1)$ is a quasi-split classical group, hence it satisfies (TWN) and (BD) by \cite{FL2},\cite{FL3}. In fact, as $\text{Spin}(3,1)\cong \SL(2,\C)$, an argument as in the proof of Theorem \ref{theorem:SL2cohomology} shows that one can do away with assuming Conjecture \ref{globalcoeffconj} as well. Finally, we have explicit results on individual cohomology groups when $G=\SO(3,1)$. Firstly, $|H^0(\overline{X(N)},\overline{L}_\rho)|=0$, and using duality arguments (see Section $5$ of \cite{MR2}),
\begin{align*}
    |H^3(\overline{X(N)},\overline{L}_\rho)|=0.
\end{align*} 
Furthermore, by a lemma of Raimbault (\cite{Raimbault2}, see also Lemma $5.2$ in \cite{MR2}), it follows that $|H^1(\overline{X(N)},\overline{L}_\rho)|=0$ is polynomially bounded in $N$. As $d = 3$, we can thus isolate the second degree of cohomology and formulate a simplified version of the theorem above.

\begin{cor}\label{cor:SO(3,1)}
    Let $n=3$, such that $G(\R)=\SO(3,1)$. Assume that $\rho_\R$ is a $\lambda$-strongly acyclic representation of $G(\R)$, and that the vector bundle $L_\rho\otimes \C$ is acyclic. Then there exists $a>0$ such that
    \begin{align*}
        \frac{\log|H^2(\overline{X(N)},\overline{L}_\rho)|}{[\Gamma:\Gamma(N)]} = -2\cdot T^{(2)}_{X(1)}(\overline{\rho}_\R)+O\left(\frac{(\log N)^a}{N^2}\right).
    \end{align*}
\end{cor}

\clearpage